\documentclass[11pt,reqno]{amsart}

\usepackage{verbatim}
\usepackage{marginnote}

\usepackage{amsmath}
\usepackage{amsfonts}
\usepackage{amssymb}
\usepackage{amsthm}
\usepackage{bbm}

\usepackage{graphicx}
\usepackage{setspace}
\usepackage{gensymb}
\usepackage{enumerate}
\usepackage{mathrsfs}
\usepackage{empheq}
\usepackage{graphics}
\usepackage{epsfig}
\usepackage[colorlinks=true, linkcolor=red, citecolor=blue]{hyperref}
\usepackage[left=2.25cm,top=3.25cm,right=2.25cm, bottom=3.1cm]{geometry}
\usepackage{graphicx, psfrag}
\usepackage{stackrel}

\DeclareMathOperator{\spt}{spt}

\DeclareMathOperator{\esssup}{ess\ sup}

\newcommand{\Lam}{\Lambda}

\newcommand{\Om}{\Omega}
\newcommand{\tht}{\theta}
\newcommand{\Tht}{\Theta}
\newcommand{\al}{\alpha}
\newcommand{\eps}{\epsilon}
\newcommand{\be}{\beta}

\newcommand{\De}{\Delta}
\newcommand{\de}{\delta}
\newcommand{\s}{\sigma}

\newcommand{\gam}{\gamma}

\newcommand{\kap}{\kappa}

\newcommand{\z}{\zeta}

\newcommand{\N}{\mathbb{N}}
\newcommand{\R}{\mathbb{R}}

\newcommand{\T}{\mathbb{T}}
\newcommand{\Z}{\mathbb{Z}}

\newcommand{\A}{\mathcal{A}}

\newcommand{\del}{\nabla}
\newcommand{\bdy}{\partial}

\newcommand{\Ri}{\mathcal{R}}

\newcommand{\til}[1]{\widetilde{#1}}

\newcommand{\lp}{\triangle}

\newcommand{\lpj}{\triangle_j}

\newcommand{\lpk}{\triangle_k}

\newcommand{\lpl}{\triangle_\ell}

\newcommand{\smod}{\setminus}

\newcommand{\sub}{\subset}

\newcommand{\ls}{\lesssim}
\newcommand{\gs}{\gtrsim}

\newcommand{\goesto}{\rightarrow}

\newcommand{\imb}{\hookrightarrow}

\newcommand{\Sch}{\mathcal{S}}

\newcommand{\Hper}{H_{{per}}^\s}
\newcommand{\Hdper}{\dot{H}_{{per}}^\s}

\newcommand{\Sob}[2]{\lVert#1\rVert_{#2}}

\newcommand{\abs}[1]{\lvert#1\rvert}

\newcommand{\lb}{\langle}
\newcommand{\rb}{\rangle}

\newcommand{\ind}{\mathbbm{1}}

\newcommand{\req}[1]{(\ref{#1})}

\newcommand{\FLp}{F_{L^p}}
\newcommand{\FHg}{F_{H^{-\gam/2}}}

\newcommand{\TLt}{\Tht_{L^2}}
\newcommand{\TLp}{\Tht_{L^p}}
\newcommand{\THs}{\Tht_{H^\s}}

\newtheorem{thm}{Theorem}

\newtheorem{prop}{Proposition}
\newtheorem{lem}[prop]{Lemma}
\newtheorem{rmk}{Remark}[section]
\newtheorem{coro}[prop]{Corollary}
\newtheorem{coro1}[thm]{Corollary}

\newtheorem*{sett}{Standing Hypotheses}

\numberwithin{equation}{section}

\title[a data assimilation algorithm for the subcritical SQG equation]
{A data assimilation algorithm for the subcritical surface quasi-geostrophic equation}

\author{Michael S. Jolly$^{1}$}
\address{$^{1}$Department of Mathematics\\
Indiana University\\ Bloomington, IN 47405}
\email[M.S. Jolly]{msjolly@indiana.edu}

\author{Vincent R. Martinez$^{2, \dagger}$}
\address{$^{2}$Department of Mathematics\\
Tulane University\\ New Orleans, LA 70118}
\email[V. R. Martinez]{vmartin6@tulane.edu}

\author{Edriss S. Titi$^{3}$}
\address{$^3$ Department of Mathematics \\
Texas A\& M University \\
College Station, TX 77843-3368\\
Also:
Department of Computer Science and Applied Mathematics\\
Weizmann Institute of Science\\
Rehovot, 76100, Israel}
\email[E.S. Titi]{titi@math.tamu.edu}

\thanks{$\dagger$ denotes corresponding author}

\begin{document}
\date{\today}
\maketitle

%

\centerline{\it  This paper is dedicated to the memory of Professor Abbas Bahri.}

\begin{abstract}
In this article, we prove that data assimilation by feedback nudging can be achieved for the three-dimensional quasi-geostrophic equation in a simplified scenario using only large spatial scale observables on the dynamical boundary.   On this boundary, a scalar unknown (buoyancy or surface temperature of the fluid) satisfies the surface quasi-geostrophic equation.  The feedback nudging is done on this two-dimensional model, yet ultimately synchronizes the streamfunction of the three-dimensional flow.  The main analytical difficulties are due to the presence of a nonlocal dissipative operator in the surface quasi-geostrophic equation.  This is overcome by exploiting a suitable partition of unity, the modulus of continuity characterization of Sobolev space norms, and the Littlewood-Paley decomposition to ultimately establish various boundedness and approximation-of-identity properties for the observation operators.



\end{abstract}


\section{Introduction and Main Results}\label{sect:intro}

Continuous data assimilation dates back to 1960s with the idea of Charney, Halem, and Jastrow in \cite{chj}, where they proposed that the equations of the atmosphere be used to process the observations, which were collected essentially continuously in time, in order to refine estimates for the current atmospheric state.  Assuming that these equations represent the observed reality exactly, one can then use the improved estimate for the current state as initial data with which to integrate the model in time and thus, provide more accurate forecasts.

The data assimilation approach studied in this article is inspired by the one proposed in \cite{azouani:titi, azouani:olson:titi}, where the large scale observations are inserted in the physical model through the addition of a feedback control term that serves to relax the the solution of the modified algorithm towards the reference solution of the original system associated with the data.  Indeed, suppose that perfect, coarse spatial scale measurements, $J_h(v)$, are given, where $h>0$ represents the spatial mesh size of the measurements and $v$ represents a reference solution that evolved from some \textit{unknown} initial value $v(0)=v_0$ according to the evolution equation
	\begin{align}\label{ref:eqn}
		\frac{dv}{dt}=F(v).
	\end{align}
In this context, one can more generally view the observation operator, $J_h$, as an interpolation operator, i.e., a finite-rank linear operator satisfying certain approximation estimates (see section \ref{sect:inter}).  Then to ``recover" $v$ forward in time, one finds an approximating solution, $w(t)$, which solves the following initial value problem:
	\begin{align}\label{gen:da}
		\frac{dw}{dt}=F(w)-\mu (J_h(w)-J_h(v)),\quad w(0)=w_0,
	\end{align}
where $w_0$ is \textit{arbitrary} and $\mu=\mu(h)>0$ is the ``nudging parameter."  The approximating solution, $w$, converges to the reference solution at an exponential rate depending on the size of $\mu$.  We remark that some of the important features of this algorithm are that it can be initialized arbitrarily, e.g. $w_0$ \textit{identically zero}, and that $J_h$ can be defined in a sufficiently general manner to accommodate observables in the form of Fourier modes, local spatial averages, nodal values, and suitable modifications of these.  Indeed, in the original algorithm of Charney, Halem, Jastrow, the observables were inserted into the nonlinear term of the equation, which made certain choices of $J_h$ difficult to implement.  We also remark that in the case of noisy observations, i.e., $J_h(v)$ replaced by $J_h(v+\dot{W})$ in \req{gen:da}, where $W$ is a white noise in time, then the resulting stochastic partial differential equation can in fact be viewed as the continuous-time limit of the 3DVAR filter (cf. \cite{blsz}).

The synchronizing model \req{gen:da} has been analyzed for several important physical models including the one-dimensional (1D) Chaffee-Infante equation, two-dimensional (2D) Navier-Stokes (NS) equation, 2D Boussinesq, three-dimensional (3D) Brinkman-Forchheimer extended Darcy equations, 3D B\'enard convection in porous media, and 3D NS $\alpha$-model, (cf. \cite{alb:lopes:titi, azouani:titi, azouani:olson:titi, farhat:jolly:titi, farhat:lunasin:titi1, farhat:lunasin:titi2, farhat:lunasin:titi3, mark:titi:trab}), while the case of contaminated measurements was studied in \cite{blsz} for measurements given by Fourier modes, and \cite{bessaih:olson:titi} for the case of more general classes of measurements, including those given as local spatial averages or nodal values.  More recently, \req{gen:da} has been adapted to accommodate insertion of spatial data which is discrete-in-time \cite{foias:mondaini:titi} or blurred-in-time by a time-averaging process \cite{jolly:martinez:olson:titi}.  It is important to note that each of these cases present their own different analytical difficulties or have structural features that allow one to reduce measurements to a certain component.  Other studies at the level of the partial differential equation (PDE) for variational or bayesian approaches to data assimilation can be found, for instance, in \cite{bllmss, hls}, while avenues to improve them were explored in, for instance, \cite{auroux:bansart:blum, law:alonso:shukla:stuart}.  Lastly, we emphasize that of the studies taken thus far for the nudging model \req{gen:da}, this article is the first to address the mathematical difficulties brought on by the appearance of a nonlocal dissipation operator whose dissipativity is weaker than that provided by the full Laplacian.  This operator appears naturally, albeit for mathematical analytical reasons, in the surface quasi-geostrophic (SQG) equation defined below in \req{sqg} (cf. \cite{desj:gren98}).  To overcome these difficulties, we establish various boundedness and approximation properties for $J_h$ in the scale of Lebesgue spaces and fractional Sobolev spaces.  We believe that the elementary nature of the estimates gives a flexibility to these properties that should prove useful in further applications or other contexts.

This particular feedback control approach to data assimilation was inspired by the fact that many dissipative systems, e.g., 2D NS equations, possess a finite number of \textit{determining parameters}, which is to say that if a projection of the difference of two solutions converge asymptotically in time to $0$, then the difference of the solutions themselves must also converge to $0$.  This property was established for the 2D NS equations in the seminal paper of Foias and Prodi \cite{foias:prodi} in the modal case and then later extended to the nodal and volume elements case in \cite{foias:temam:nodes, jones:titi:nodes, jones:titi:volelts, jones:titi}.  The case of general determining interpolants was introduced and investigated in \cite{cjt1, cjt2}.  

The physical model of interest in this article is a special case of the 3D quasi-geostrophic (QG) equation on a half-space.  The 3D QG equation asserts the conservation of ``potential vorticity" subject to a dynamical boundary condition.  It is valid in the regime of strong rotation, where the time scales associated with atmospheric flow over long distances are much larger than the time scales associated with the Earth's rotation, i.e., low Rossby number.  It is the simplest model of such fluids with nontrivial dynamics that describes the departure from geostrophically balanced flows, i.e., where the Coriolis force and the horizontal pressure gradient are in balance (cf. \cite{pedlosky}).  An interesting consequence of this balance is that it imposes a planar dimensionality on the corresponding solution, and is thus the source for the strong 2D features of the otherwise 3D flow (cf. \cite{const:qg:spectra}).  In the simplified scenario where the potential vorticity, $q$, is advected along material lines, one has that $q\equiv 0$ solves its evolution equation exactly, but retains the evolution of its streamfunction through its boundary condition.  Stated in \textit{non-dimensionalized} variables and in terms of the potential vorticity streamfunction, $\Psi=\Psi(\mathbf{x}, t)$, $\mathbf{x}=(x,y,z)$, we have
	\begin{align}\label{stream}
		\De_{3D}\Psi=0,\quad\text{in}\ \Om\times\{z>0\}\times\{t>0\},\quad\frac{\bdy}{\bdy z}\Psi(\cdotp,0,t)=\tht(\cdotp,t),\quad \lim_{z\goesto\infty}\Psi(\cdotp, z,t)=0,
	\end{align}
where $\De_{3D}=\bdy_x^2+\bdy_y^2+\bdy_z^2$ and the Neumann boundary condition, $\tht$, satisfies the forced, dissipative SQG equation:
\begin{align}\label{sqg}
			\bdy_t\tht+\kap\Lam^\gam\tht+u\cdotp{\del}\tht=f,\quad u=\Ri^\perp\tht,
	\end{align}
and both  $\Psi$ and $\tht$ are subjected to periodic boundary conditions in the horizontal variables $x,y$ over the fundamental periodic domain $\Om=[-\pi,\pi]^2$.  Here, $\gam\in(0,2)$ and $\tht$ represents the scalar surface temperature or buoyancy of a fluid, which is advected along the velocity vector field, $u$, and $f$ is a given source term.  The external forcing can be time-dependent provided that it is bounded in time with values in the relevant spatial space; for simplicity, we will deal with the case of $f$ being time-independent.  The velocity is related to $\tht$ by a Riesz transform, $\Ri^\perp:=(-R_2,R_1)$, where the symbol of $R_j$ is given by $i\xi_j/\abs{\xi}$.  For $0<\gam\leq2$, we denote the dissipative operator by $\Lam^\gam:=(-\De)^{\gam/2}$, which is the operator whose symbol is $\abs{\xi}^\gam$; it appears with $\gam=1$ when accounting for viscous drag at the boundary, i.e., Ekman pumping, (cf. \cite{const:qg:spectra, desj:gren98}).  In this paper, we will concern ourselves with the \textit{subcritically diffusive} case, $\gam\in(1,2)$, which still preserves the analytical difficulties that we will encounter arising from the nature of its nonlocality, but simplifies the issue of well-posedness for the sychronizing model \req{fb:sqg}.  We will assume that $\tht_0, f$ are $\pi$-periodic in $\Om$ with mean zero.  In particular, $\tht$ is also $\pi$-periodic with mean zero since the mean-zero property is preserved by the evolution.


While it is still not fully understood how to describe the motion of fluids in the regime of low Rossby number, the SQG equation has many interesting features which are relevant both physically and mathematically \cite{held}.  For instance, experiments \cite{qg:spectra:experiments} have shown that the energy spectra of rotating fluids exhibit a power law scaling different from the one derived through the 2D NS equations, but consistent with the one derived through the SQG equation \cite{const:qg:spectra}.  On the other hand, mathematically, the SQG equation has striking resemblance to the 3D NS equations.  Indeed, by applying $\del^\perp$ to \req{sqg} one uncovers a mechanism for vortex stretching analogous to the one found in the vorticity formulation of the 3D NS equations.  Thus, the SQG equation provides a two-dimensional scalar model which exhibits three-dimensional phenomena and challenges.  Since its introduction into the mathematical community by Constantin, Majda, and Tabak \cite{cmt}, the inviscid equation, along with its critical ($\gam=1$) and supercritical ($\gam<1$) counterparts, have been extensively studied, and by now, well-posedness in various function spaces and global regularity has been resolved in all but the supercritical case (cf. \cite{caff:vass, const:vic, const:wu:qgwksol, ctv, cz:vic, hdong, kis:naz, kis:naz:vol, resnick}).  The long-time behavior and existence of a global attractor has also been studied in both the subcritical and critical cases (cf. \cite{carrillo:ferreira, chesk:dai:qgattract, const:coti:vic, coti:zelati, ctv, ju:qgattract}).  In particular, in \cite{ju:qgattract}, Ju established the existence of a global attractor, $\A$, for \req{sqg} in the subcritical regime $\gam\in(1,2)$, which we recall in section \ref{sect:ball}.  In the context of the nudging scheme induced by \req{fb:sqg}, we will assume that the noiseless observables are sampled from an absorbing ball that contains $\A$ (see section \ref{sect:a priori}).

In this paper, data assimilation achieved through
	\begin{align}\label{fb:sqg}
			\bdy_t\eta+\kap\Lam^\gam\eta+v\cdotp{\del}\eta=f-\mu J_h(\eta-\tht),\quad v=\Ri^\perp\eta,
	\end{align}
subject to periodic boundary conditions over $\Om=[-\pi,\pi]^2$, as in \req{sqg}.  We recall that $\mu$ is the nudging parameter and $J_h$ with $h>0$ is an interpolating operator based on coarse spatial measurements that satisfies certain approximation properties.  We will first establish in Theorem \ref{thm1} the global existence and uniqueness of strong solutions to \req{fb:sqg} in a Sobolev class of functions for certain time-independent source terms, provided that $\mu, h$ are  {jointly} chosen appropriately.  The estimates required to prove Theorem \ref{thm1} will be performed in section \ref{sect:a priori}.
We then show that if $\mu$ is chosen large enough and $h$ is correspondingly taken sufficiently small, then the unique solution of \req{fb:sqg} synchronizes at an exponential rate with the reference solution, $\tht$, 
for any $\gam\in(1,2)$  (see section \ref{sect:inter}).  In particular, we show that 
	\begin{align}\label{intro:synch}
		\Sob{\eta(t)-\tht(t)}{L^2(\Om)}\goesto0,\quad\text{at an exponential rate as}\ t\goesto\infty.
	\end{align}
Convergence of the corresponding \textit{three-dimensional} streamfunctions then follows immediately from this (Corollary \ref{coro:stream}).  To see this, let $\Psi$ be given by \req{stream} and $\Psi_\eta$ be given by
	\begin{align}\label{stream:eta}
		\De_{3D}\Psi_\eta=0,\quad\text{in}\ \Om\times\{z>0\}\times\{t>0\},\quad\frac{\bdy}{\bdy z}\Psi_{\eta}(\cdotp,0,t)=\eta(\cdotp,t),\quad \lim_{z\goesto\infty}\Psi_\eta(\cdotp,z,t)=0.
	\end{align}
Then \req{intro:synch} implies
	\begin{align}\notag
		\Sob{\del(\Psi_\eta(t)-\Psi(t))}{L^2(\Om\times\R_+)}\goesto0,\quad\text{exponentially as}\ t\goesto\infty.
	\end{align}
Indeed, since the Dirichlet-to-Neumann map determines the relations $\Psi|_{z=0}=-\Lam^{-1}\tht$ and $\Psi|_{z=0}=-\Lam^{-1}\eta$, upon taking the $L^2$-scalar product of \req{stream} with $\Psi$, integrating by parts, and applying the boundary conditions it follows that
	\begin{align}\label{stream:synch}
		\Sob{\del(\Psi(t)-\Psi_\eta(t))}{L^2(\Om\times\R_+)}^2= \left|\int_{\Om}(\eta(t)-\tht(t))\Lam^{-1}(\eta(t)-\tht(t))\ dx dy\right|\leq C\Sob{\eta(t)-\tht(t)}{L^2(\Om)}^2,
	\end{align}
for some absolute constant $C>0$.  We refer to \cite{sobolev}, for example, for details regarding the Neumann problem.  One may thus conclude, at least in the simplified scenario described above, that  measurements \textit{on only the boundary} $\Om\times\{z=0\}\sub\R^3_+$ are required for the synchronization of the streamfunction over the entire 3D domain \req{fb:sqg}.  Studies on data assimilation on simple forecast models, in which some state variable observations are not available as input were carried out in \cite{ghil:shkoller:yangarber, ghil:halem:atlas}, for instance.  It was observed that although the nonlinearity in the models can mediate couplings across all length scales, the full state of the system can nevertheless be recovered by employing coarse mesh measurements of only some the state variables so long as one uses a good dynamical model and data assimilation algorithm (cf. \cite{chj, farhat:lunasin:titi3}).  Our result, therefore, rigorously confirms this understanding in a precise way, e.g., employing only surface measurements to recover the three-dimensional stream function,  through the SQG model \req{sqg} with the corresponding nudging equation \req{fb:sqg}.

\section{Preliminaries}\label{sect:prelim}


\subsection{Function spaces: $L^p_{per}$, $V_\s$, $H_{{per}}^\s$, $\dot{H}_{{per}}^\s$, $C^\infty_{per}$}\label{hs:def}
Let $1\leq p\leq\infty$, $\s\in \R$ and $\T^2=\R^2/(2\pi\Z)=[-\pi,\pi]^2$.  Let $\mathcal{M}$ denote the set of real-valued Lebesgue measureable functions over $\T^2$.  Since we will be working with periodic functions, we define
	\begin{align}
		\mathcal{M}_{per}:=\{\phi\in\mathcal{M}:\phi(x,y)=\phi(x+2\pi,y)=\phi(x,y+2\pi)=\phi(x+2\pi,y+2\pi)\ \text{a.e.}\}.
	\end{align}
Let $C^\infty(\T^2)$ denote the class of functions which are infinitely differentiable over $\T^2$.  We define $C^\infty_{per}(\T^2)$ by
	\begin{align}\notag
		C^\infty_{per}(\T^2):=C^\infty(\T^2)\cap\mathcal{M}_{per}.
	\end{align} 
For $1\leq p\leq\infty$, we define the periodic Lebesgue spaces by
	\begin{align}\notag
		L^p_{per}(\T^2):=\{\phi\in\mathcal{M}_{per}: \Sob{\phi}{L^p}<\infty\},
	\end{align}
where
	\begin{align}\notag
		\Sob{\phi}{L^p}:=\left(\int_{\T^2} |\phi(x)|^p\ dx\right)^{1/p},\quad 1\leq p<\infty,\quad\text{and}\quad \Sob{\phi}{L^\infty}:=\ \stackrel[x\in\T^2]{}{\esssup}|\phi(x)|.
	\end{align}
Let us also define
	\begin{align}\label{Z:space}
		\mathcal{Z}:=\{\phi\in L^1_{per}:\int_{\T^2} \phi(x)\ dx=0\}.
	\end{align}

Let $\hat{\phi}(\mathbf{k})$ denote the Fourier coefficient of $\phi$ at wave-number $\mathbf{k}\in\Z^2$.  For any real number $\s\in\R$, we define the homogeneous Sobolev space, $\Hdper(\T^2)$, by
	\begin{align}\label{Hdper}
		\Hdper(\T^2):=\{\phi\in L^2_{per}(\T^2):\quad \Sob{\phi}{\dot{H}^\s}<\infty\},
	\end{align}
where
	\begin{align}\label{hdper:norm}
		\sum_{\mathbf{k}\in\Z^2\smod\{\mathbf{0}\}}\abs{\mathbf{k}}^{2\s}|\hat{\phi}(\mathbf{k})|^2.
	\end{align}
We define the inhomogeneous Sobolev space, $\Hper(\T^2)$, by
	\begin{align}\label{Hper}
		\Hper(\T^2):=\{\phi\in L^2_{per}(\T^2):\Sob{\phi}{H^\s}<\infty\},
	\end{align}
where
	\begin{align}\label{hper:norm}
		\Sob{\phi}{{H}^\s}^2:=\sum_{\mathbf{k}\in \Z^2}(1+\abs{\mathbf{k}}^2)^{\s}|\hat{\phi}(\mathbf{k})|^2.
	\end{align}

Let $\mathcal{V}_0\sub\mathcal{Z}$ denote the set of trigonometric polynomials with mean zero over $\T^2$ and set
	\begin{align}\label{hper}
		V_\s:=\overline{\mathcal{V}_0}^{H^\s},
	\end{align}
where the closure is taken with respect to the norm given by \req{hper:norm}.  
Observe that the mean-zero condition can be equivalently stated as $\hat{\phi}(\mathbf{0})=0$.  Thus, $\Sob{\cdotp}{\dot{H}^\s}$ and $\Sob{\cdotp}{H^\s}$ are equivalent as norms over $V_\s$.  Moreover, by  Plancherel's theorem we have
	\begin{align}\notag
		\Sob{\phi}{\dot{H}^\s}=\Sob{\Lam^\s \phi}{L^2}.
	\end{align}
Finally, for $\s\geq0$, we identify $V_{-\s}$ as the dual space, $(V_\s)'$, of $V_\s$, which can be characterized as the space of all bounded linear functionals, $\psi$, on $V_\s$ such that
	\[
		\Sob{\psi}{\dot{H}^{-\s}}<\infty.
	\]
Therefore, we have the following continuous embeddings
	\[
		V_\s\imb V_{\s'}\imb V_0\imb V_{-\s'}\imb V_{-\s},\quad 0\leq\s'\leq\s.
	\]

\begin{rmk}
Since we will be working over $V_\s$ and $\Sob{\cdotp}{\dot{H}^\s}$, $\Sob{\cdotp}{H^\s}$ determine equivalent norms over $V_\s$, we will often denote $\Sob{\cdotp}{\dot{H}^\s}$ simply by $\Sob{\cdotp}{H^\s}$ for convenience.  We will distinguish between the inhomogeneous and homogeneous Sobolev norm when we are outside of this context (see Appendix).
\end{rmk}

\subsection{Interpolant Observables}\label{sect:inter}
We will consider two types of interpolant observables, which we refer to as Type I and Type II.  We show in the Appendix that the interpolant given by local averages of the function over cubes that partition the domain, $\T^2$, are of Type I, and that modal projection onto finitely many, fixed wave-numbers are of Type II.

Let $h>0$, $1\leq p\leq\infty$, and $J_h:{L}^p(\T^2)\goesto {L}^p(\T^2)$ be a linear operator.  Suppose that $J_h$ satisfies
	\begin{align}
		&\sup_{h>0}\Sob{J_h\phi}{L^p}\leq C\Sob{\phi}{L^p}, \quad 1< p<\infty,\label{Ih:base0}\\
		&\Sob{J_h\phi}{L^p}\leq Ch^{2/p-1}\Sob{\phi}{L^2},\quad p\geq2\label{Ih:base1},\\
		&\Sob{J_h\phi}{\dot{H}^\be}\leq Ch^{-\be}\Sob{\phi}{L^2},\quad\be\geq0\label{Ih:base2},
	\end{align}
where $C>0$ is an absolute constant independent of $\phi$.  Let us remark here that in our main theorems, we will impose that $h\leq\bar{h}$ for some $\bar{h}>0$, where $\bar{h}$ is the size of the required spatial resolution of the collected measurements  (see section \ref{sect:a priori}).  Thus, at least within the context of the theorems, only the operators $J_h$ for $h\leq\bar{h}$ are considered, though numerical studies of the algorithm suggest that synchronization can occur for values of $\bar{h}$ far larger than what is suggested by the analytical bounds given by the theorems (cf. \cite{hoteit, gesho:olson:titi}).

\subsubsection*{Type I}
In addition to \req{Ih:base0}-\req{Ih:base2}, interpolants of Type I will also satisfy
	\begin{align}\label{TI}
		\Sob{\phi-J_h\phi}{L^2}\leq Ch^{\be}\Sob{\phi}{\dot{H}^\be}\quad\text{and}\quad\Sob{\phi-J_h\phi}{\dot{H}^{-\be}}\leq Ch^{\be}\Sob{\phi}{L^2},\quad \be\in(0,1).
	\end{align}

\subsubsection*{Type II}
In addition to \req{Ih:base0}-\req{Ih:base2}, interpolants of Type II will also satisfy
	\begin{align}\label{TII}
		\Sob{\phi-J_h\phi}{\dot{H}^\al}\leq Ch^{\be-\al}\Sob{\phi}{\dot{H}^\be},\quad \be>\al\quad\text{and}\quad \Lam^\be J_h\phi=J_h\Lam^\be\phi.
	\end{align}

Observe that Type II interpolants are also Type I.  We refer to the Appendix, where we provide examples of both Type I and II interpolants.

\subsection{Inequalities for fractional derivatives}

We will make use of the following bound for the fractional Laplacian, which can be found for instance in \cite{const:gh:vic, ctv, ju:qgattract}.

\begin{prop}\label{lb}
Let $p\geq2$, $0\leq\gam\leq2$, and $\phi\in C_{per}^\infty(\T^2)$. Then 
	\begin{align}\notag
		\int_{\T^2}|\phi|^{p-2}(x)\phi(x)\Lam^\gam \phi(x)\ dx\geq \frac{2}p\Sob{\Lam^{\gam/2}(|\phi|^{p/2})}{L^2}^2.
	\end{align}
\end{prop}

We will  also make use of the following calculus inequality for fractional derivatives (cf. \cite{kato:ponce, kenig:ponce:vega} and references therein):

\begin{prop}\label{prod:rule}
Let $\phi, \psi\in C_{per}^\infty(\T^2)$, $\be>0$, and $p\in(1,\infty)$.  Let $1/p=1/p_1+1/p_2=1/p_3+1/p_4$, and $p_2,p_3\in(1,\infty)$.  There exists an absolute constant $C>0$, depending only on $\s, p, p_i$, such that
	\begin{align}\notag
		\Sob{\Lam^\be(\phi\psi)}{L^p}\leq C\Sob{\psi}{L^{p_1}}\Sob{\Lam^\be \phi}{L^{p_2}}+C\Sob{\Lam^\be \psi}{L^{p_3}}\Sob{\phi}{L^{p_4}}.
	\end{align}
\end{prop}

Finally, we will frequently apply the following interpolation inequality, which is a special case of the Gagliardo-Nirenberg interpolation inequality and can be proven by using the Plancherel theorem combined with H\"older's inequality:

\begin{prop}\label{interpol}
Let $\phi\in {H}^{\be}_{per}(\T^2)$ and $0\leq \al< \be$.
There exists an absolute constant $C>0$, depending only on $\al, \be$ such that
	\begin{align}\label{gn:ineq}
		\Sob{\Lam^{\al} \phi}{L^2}\leq C\Sob{\Lam^{\be} \phi}{L^2}^{\frac{\al}{\be}}\Sob{\phi}{L^2}^{1-\frac{\al}{\be}}.
	\end{align}
\end{prop}

\subsection{$L^p$ bounds and Global Attractor of SQG equation}\label{sect:ball}
Let us recall the following estimates for the reference solution $\tht$ (cf. \cite{ctv, ju:qgattract, resnick}).

\begin{prop}\label{prop:sqg:ball}
Let $\kap>0$, $\gam\in[0,2]$ and $\tht_0, f\in {L}^p_{per}(\T^2)\cap\mathcal{Z}$. 
Suppose that $\tht\in {L}_{per}^p(\T^2)$ is a smooth solution of \req{sqg} such that $\tht(\cdotp, 0)=\tht_0(\cdotp)$.  There exists an absolute constant $C>0$ such that for any $p\geq2$, we have
	\begin{align}\label{fp}
		\Sob{\tht(t)}{L^p}\leq \left(\Sob{\tht_0}{L^p}-\frac{1}C{\FLp}\right)e^{-C{{{\kap}}}t}+\frac{1}{C}{\FLp},\quad {\FLp}:=\frac{1}{{\kap}}\Sob{f}{L^p}.
	\end{align}
Moreover, for $p=2$ and $f\in V_{-\gam/2}$, we have
	\begin{align}\label{fgam}
		\Sob{\tht(t)}{L^2}^2\leq \left(\Sob{\tht_0}{L^2}^2-{\FHg^2}\right)e^{-{{{\kap}}} t}+{\FHg^2}, \quad \FHg:=\frac{1}{{\kap}}\Sob{f}{H^{-\gam/2}}.
	\end{align}
\end{prop}

It was shown in \cite{ju:qgattract} that in the subcritical range, $1<\gam\leq2$, equation \req{sqg} has an absorbing ball in ${H}_{per}^\s(\T^2)$, for $\s>2-\gam$, and corresponding global attractor, $\A\sub V_\s$.  We recall that an $H^r_{per}(\T^2)$-absorbing set for a dissipative equation is a bounded set $\mathcal{B}\sub H^r_{per}(\T^2)$ characterized by the property that for any $\tht_0\in H^r_{per}(\T^2)$, there exists $t_0=t_0(\Sob{\tht_0}{H^r_{per}})>0$ such that $S(t)\tht_0\in\mathcal{B}$ for all $t\geq t_0$, where $\{S(t)\}_{t\geq0}$ denotes the semigroup of the corresponding dissipative equation.

\begin{prop}[Global attractor]\label{prop:ga}
Suppose that $1<\gam\leq2$ and  $\s>2-\gam$.  Let $f\in V_{\s-\gam/2}\cap {L}_{per}^p(\T^2)$, where $1-\s<2/p<\gam-1$.  Then \req{sqg} has an absorbing ball $\mathcal{B}_{H^\s}$ given by
	\begin{align}\label{sqg:hs:ball}
		\mathcal{B}_{H^\s}:=\{\tht\in V_\s: \Sob{\tht}{H^\s}\leq \THs\},
	\end{align}
for some $\THs<\infty$.  Moreover, the solution operator $S(t)\tht_0=\tht(t)$, $t>0$ of \req{sqg} defines a semigroup in the space $V_\s$ and possesses a global attractor $\A\sub V_\s$, i.e., $\A$ is a compact, connected subset of $V_\s$ satisfying the following properties:
	\begin{enumerate}
		\item $\A$ is the maximal bounded invariant set;
		\item  $\A$ attracts all bounded subsets in $V_\s$ in the topology of $\dot{H}_{{per}}^\s$.
	\end{enumerate}
\end{prop}

Before we move on to the a priori analysis, we will set forth the following convention for constants.
\begin{rmk}
In the estimates that follow below, $c, C$, will denote generic positive absolute constants, which depend only on other non-dimensional scalar quantities, and may change line-to-line in the estimates.   We also use the notation $A\lesssim B$ and $A\sim B$ to denote the relations $A\leq cB$ and $c'B\leq A\leq c''B$, respectively, for some absolute constants $c, c', c''>0$.
\end{rmk}

\section{Main Results and A Priori Estimates}\label{sect:a priori}
We will work in the following setting throughout this section.

\begin{sett}
Assume the following:
\begin{enumerate}[(H1)]
	\item $1<\gam\leq 2$;
	\item $\s>2-\gam$;
	\item $p\in[1,\infty]$ such that $1-\s<2/p<\gam-1$, fixed;
	\item $\eta_0\in V_\s$
	\item $f\in V_{\s-\gam/2}\cap L^{{p}}$, time-independent;
	\item $\tht_0\in\mathcal{B}_{H^\s}\cap L^{{p}}$, where $\mathcal{B}_{H^\s}$ is the $H^\s$-absorbing ball with radius $\Tht_{H^\s}$ from Proposition \ref{prop:ga};
	\item $J_h$ satisfies \req{Ih:base0}, \req{Ih:base1}, \req{Ih:base2} and is either of Type I or Type II.
\end{enumerate}
\end{sett}
Observe that by Proposition \ref{prop:sqg:ball}, $(H3)$ and $(H5)$ immediately imply that the solution $\tht$ of \req{sqg} corresponding to initial data $\tht_0$ satisfies
	\begin{align}\label{tht:bounds}
		\TLt:=\sup_{t>0}\Sob{\tht(t)}{L^2}<\infty\quad \text{and}\quad\TLp:=\sup_{t>0}\Sob{\tht(t)}{L^p}<\infty.
	\end{align}

We will first show (in section \ref{a priori:step1}) that smooth solutions to \req{fb:sqg}, $\eta(t)$, satisfy
	\begin{align}\label{eta:l2}
		M_{L^2}:=\sup_{t>0}\Sob{\eta(t)}{L^2}<\infty,
	\end{align}
and (in section \ref{l2:to:lp}) that this implies 	
	\begin{align}\label{eta:lp}
		M_{L^p}:=\sup_{t>0}\Sob{\eta(t)}{L^p}<\infty.
	\end{align}
These two bounds  will then be used to show (in section \ref{sect:hs1}) that
	\begin{align}\label{eta:hs}
		M_{H^{\s}}:=\sup_{t>0}\Sob{\eta(t)}{{H}^\s}<\infty.
	\end{align}
With these estimates in hand and \textit{under the Standing Hypotheses}, we establish in section \ref{sect:pf1} short-time existence and uniqueness in the space ${H}^\s$ in section \ref{sect:pf1}:

\begin{thm}\label{thm1}
Assume that $(H1)-(H7)$ holds.  Let $\tht$ be the unique global strong solution of \req{sqg} corresponding to $\tht_0$.  There exists $\rho=\rho(h,\s,\gam)$ (given by \req{t:de} below) such that if
	\begin{align}\label{num:modes}
		\quad\frac{\mu}{\kap} \rho(h, \s, \gam)\lesssim1,
	\end{align}
then for each $T>0$, there exists a unique strong solution $\eta\in L^\infty(0,T;V_\s)\cap L^2(0,T;\dot{H}_{{per}}^{\s+\gam/2})$ of \req{fb:sqg} such that
	\begin{align}\notag
		\Sob{\eta(t)}{H^\s}\lesssim M_{H^\s},\quad t\in [0,T],
	\end{align}
for some quantitity $M_{H^\s}$ (given by \req{Ms} below) that depends only on $\mu, \kap$, $\Sob{f}{H^{\s-\gam/2}}$, and $\Tht_{H^\s}$.  Moreover, $\eta\in C([0,T]; V_{\s-\eps})$ for all $\eps\in(0,\s+1/2)$.
\end{thm}

Ultimately, the estimates we collect will also be used to ensure asymptotic synchronization of $\eta$ to the reference solution $\tht$.

\begin{thm}\label{synch}
Assume that $(H1)-(H7)$ holds.  Let $\tht$ be the unique global strong solution of \req{sqg} corresponding to $\tht_0$. 
Suppose that $\mu$ satisfies
	\begin{align}\label{mu:assumpt}
		 \frac{\mu}{\kap}\gs\left(\frac{\Tht_{L^p}}{\kap}\right)^{\gam/(\gam-1-2/p)}\quad.
	\end{align}
There exists an absolute constant $c_\s<1$, depending on $\s$, such that if $\mu, h$ satisfy
	\begin{align}\label{num:modes2}
		{\mu h^\gam}\ls\kap,
	\end{align}
then 
	\begin{align}\label{synch:high1}
		\Sob{\eta(t)-\tht(t)}{L^2}^2\lesssim O(e^{-c_\s\mu t}),\quad t>0,
	\end{align}
where $\eta$ is the unique global strong solution to \req{fb:sqg} corresponding to $\eta_0$
\end{thm}

As we discussed at the end of section \ref{sect:intro}, Theorem \ref{synch} immediately implies the synchronization of the streamfunctions corresponding to $\tht$ and $\eta$.

\begin{coro1}\label{coro:stream}
Assume that $(H1)-(H7)$ holds.  Let $\tht, \eta$ be the unique global strong solutions of \req{sqg}, \req{fb:sqg} correpsponding to $\tht_0, \eta_0$, respectively.  Let $\Psi, \Psi_\eta$ denote the corresponding streamfunctions of $\tht, \eta$, i.e. satisfying \req{stream}, \req{stream:eta}, respectively.  Suppose that $\mu$ satisfies \req{mu:assumpt}.  Then, there exists an absolute constant $C>0$ such that if $\mu, h$ satisfying \req{num:modes2}, then
	\begin{align}
		\Sob{\del(\Psi_\eta(t)-\Psi(t))}{L^2(\Om\times\R_+)}\ls O(e^{-c_\s\mu t}),
	\end{align}
\end{coro1}

\begin{rmk}
We point out that to guarantee existence and uniqueness of strong solutions to \req{fb:sqg}, it suffices for $h, \mu$ to satisfy
	\begin{align}\label{uniq:cond}
		{\mu}h^\gam\lesssim \kap.
	\end{align}
To guarantee the synchronization property, \req{synch:high1}, it suffices for $h,\mu$ to satisfy \req{mu:assumpt} in addition to \req{uniq:cond}.
\end{rmk}

\begin{rmk}\label{rmk:opt}
Note that in the case where $J_h$ is given by projection onto wave numbers of size $\leq{N}$ (see section \ref{sect:inter} and Appendix), the conditions \req{mu:assumpt} with \req{num:modes} provide an estimate on the number of modal observables that are sufficient for the algorithm to synchronize to the reference solution.  We point out that this bound matches the scaling for the number of determining modes for  \req{sqg} obtained in \cite{chesk:dai:subcritqg}.  Indeed, denote by $\mathcal{M}$ the smallest number with the property that if the difference of the modes up to size ${\mathcal{M}}$ converge to $0$, as $t\goesto\infty$, then so must the difference of the solutions themselves. 
It is shown in \cite{chesk:dai:subcritqg} that
	\begin{align}\label{dm:chesk:dai}
		{\mathcal{M}}\lesssim\left(\frac{\Tht_\infty}{\kap}\right)^{1/(\gam-1)},
	\end{align}
where $\Tht_\infty$ is defined as the smallest constant for which $\{\Sob{\tht}{L^\infty}\leq\Tht_\infty\}$ is an absorbing ball in $L^\infty$ for \req{sqg}.

On the other hand, if $J_h$ is given by projection onto finitely many Fourier modes, i.e., of Type II, we may take $\rho(h,\s,\gam)=h^\gam$, where $h=2\pi/N$, so that (formally) setting $p=\infty$ in Theorem \ref{synch}, the  resulting condition on $N$ becomes (see \req{synch:bal1} and \req{synch:xi})
	\begin{align}\notag
		{{N}}\gtrsim \left(\frac{\Tht_\infty}{\kap}\right)^{1/(\gam-1)}.
	\end{align}
We note that for us, this choice for $p$ is valid when $\s>1$.

In the case where $J_h$ is given by local spatial averages over cubes of side-length $\sim h$, we may also take $\rho(h,\gam)=h^\gam$, when $\s\geq\gam/2$, so that (at least for $\s>1$), the number of local averages required to guarantee synchronization is proportional to $\Tht_\infty^{1/(\gam-1)}$.
\end{rmk}

We will perform energy estimates on the solutions to the following initial value problem given by
	\begin{align}\label{fb:ivp}
			\bdy_t\eta+{{\kap}}\Lam^\gam\eta+v\cdotp{\del}\eta=f-{{\mu}}J_{{h}}(\eta-\tht),\quad v=\Ri^\perp\eta, \quad \eta(x,0)=\eta_0(x),
	\end{align}
where ${{\kap}}$ is defined as in \req{sqg},  $h,\mu$ are positive, absolute constants whose magnitudes are to be specified later, 
and where $\tht$ is a solution to \req{sqg} corresponding to $\tht_0\in\mathcal{B}_\s$.

We note that in what follows, the estimates we perform are formal, though they may be done rigorously at the level of the equation with artificial viscosity, i.e., \req{fb:sqg} with the additional term $-\nu\De\eta$ on the left-hand side (see section \ref{sect:pf1} for details).   We emphasize that the estimates we obtain will be independent of $\gam$ and we pass to the limit as $\nu\goesto0$.

\subsection{Uniform $L^2$ estimates}\label{a priori:step1}

For $\mu>0$, let $ {\FHg}$ be given by \req{fgam}, and $\Tht_{L^2}$ by \req{tht:bounds}.  Define
	\begin{align}\label{Rveps}
		R_{L^2}^2:=C\left(\frac{{{\kap}}}{{{\mu}}}{\FHg^2}+\TLt^2\right),
	\end{align}
where $C>0$ is an absolute constant.

\begin{prop}\label{prop:a priori0}
There exist absolute constants $c_0, C_0>0$ with $c_0$ depending on $C_0$ such that if \req{Rveps} holds with $C=C_0$ and if $\mu, h$ satisfy
	\begin{align}\label{l2:modes}
		\frac{{{\mu}}h^\gam}{{{\kap}}}\leq c_0,
	\end{align}
then the following inequalities hold:
	\begin{align}\label{m2:3}
		\Sob{\eta(t_2)}{L^2}^2+{{\kap}}\int_{t_1}^{t_2}e^{-{{\mu}} (t_2-s)}\Sob{\eta(s)}{H^{\gam/2}}^2\ ds\leq \left(\Sob{\eta_0}{L^2}^2- R_{L^2}^2\right)e^{-{{\mu}} (t_2-t_1)}+ R_{L^2}^2,
	\end{align}
and
	\begin{align}\label{l2:energy}
		\Sob{\eta(t_2)}{L^2}^2+\kap\int_{t_1}^{t_2}\Sob{\eta(s)}{H^{\gam/2}}^2\ ds\leq \Sob{\eta(t_1)}{L^2}^2+\mu R_{L^2}^2(t_2-t_1),
	\end{align}
for all $0\leq t_1<t_2$.
In particular, we have
	\begin{align}\label{M2b}
		M_{L^2}(t_1,t_2)\leq\Sob{\eta(t_1)}{L^2}+R_{L^2},
	\end{align}
where $M_{L^2}(t_1,t_2):=\sup_{t\in[t_1,t_2]}\Sob{\eta(t)}{L^2}$.
\end{prop}

\begin{proof}

We multiply \req{fb:ivp} by $\eta$ and integrate over $\T^2$ to write
	\begin{align}
		\frac{1}2\frac{d}{dt}\Sob{\eta}{L^2}^2+{{\kap}}\Sob{\Lam^{\gam/2}\eta}{L^2}^2+{{{\mu}}}\Sob{\eta}{L^2}^2
			&\leq \int f\eta\ dx+{{{\mu}}}\int(\eta-J_h\eta)\eta\ dx+{{{\mu}}}\int J_h\tht \eta\ dx\notag\\
			&= I+II+III\notag.
	\end{align}
Observe that by the Cauchy-Schwarz and Young's inequalities, as well as \req{Ih:base0}, we have
	\begin{align}\notag
	\begin{split}
		I&\leq\frac{1}{{{{\kap}}}}\Sob{\Lam^{-\gam/2}f}{L^2}^2+\frac{{{{\kap}}}}{4}\Sob{\Lam^{\gam/2}\eta}{L^2}^2\\
		III&\leq {{{\mu}}}\Sob{J_h\tht}{L^2}\Sob{\eta}{L^2}\leq C{{{\mu}}}\Sob{\tht}{L^2}\Sob{\eta}{L^2}\leq C{{{{\mu}}}}\Sob{\tht}{L^2}^2+\frac{{{{\mu}}}}{4}\Sob{\eta}{L^2}^2.
	\end{split}
	\end{align}
On the other hand, by \req{TI} or \req{TII} we have
	\begin{align}\notag
	 	II&\leq {{{\mu}}}\Sob{\eta-J_h\eta}{L^2}\Sob{\eta}{L^2}\leq C{{{\mu}}}h^{\gam/2}\Sob{\Lam^{\gam/2}\eta}{L^2}\Sob{\eta}{L^2}\leq \frac{{{{\kap}}}}4\Sob{\Lam^{\gam/2}\eta}{L^2}^2+C\frac{h^\gam{{{\mu}}}^2}{{{{\kap}}}}\Sob{\eta}{L^2}^2.
	\end{align}

Combining these estimates with \req{tht:bounds} yields
	\begin{align}\label{k=0:est}
		\frac{1}2\frac{d}{dt}\Sob{\eta}{L^2}^2+\frac{{{{\kap}}}}2&\Sob{\Lam^{\gam/2}\eta}{L^2}^2+{{{{\mu}}}}\left(\frac{3}4-Ch^\gam\frac{{{{\mu}}}}{{{{\kap}}}}\right)\Sob{\eta}{L^2}^2\leq\frac{1}{{{{\kap}}}}\Sob{\Lam^{-\gam/2}f}{L^2}^2+{C}{{{\mu}}}\TLt^2.
	\end{align}
Upon applying \req{l2:modes} to \req{k=0:est}, 
we arrive at
	\begin{align}\label{k=0:est2}
		\frac{d}{dt}\Sob{\eta}{L^2}^2+{{{\mu}}}\Sob{\eta}{L^2}^2+{{\kap}} \Sob{\eta}{H^{\gam/2}}^2\leq 2\left({{\kap}}{\FHg^2} +C{{{\mu}}}\TLt^2\right).
	\end{align}
Hence, Gronwall's inequality and \req{Rveps} imply that
	\begin{align}\notag
		\Sob{\eta(t_2)}{L^2}^2+{{{\kap}}}\int_{t_1}^{t_2}e^{-{{{\mu}}}(t_2-s)}\Sob{\eta(s)}{H^{\gam/2}}^2\ ds\leq& \Sob{\eta(t_1)}{L^2}^2e^{-{{{\mu}}}(t_2-t_1)}+2R_{L^2}^2(1-e^{-{{{\mu}}}(t_2-t_1)}).
	\end{align}
On the other hand, integrating \req{k=0:est2} over $[t_1, t_2]$ gives \req{l2:energy} as desired.  This completes the proof.
\end{proof}

\subsection{$L^2$ to $L^p$ uniform bounds}\label{l2:to:lp}
Let $p>2$ and $\mu>0$.  Let  $\FLp$ be given by \req{fp}, $\Tht_{L^p}$ by \req{tht:bounds}, and $M_{L^p}$ by \req{eta:l2}.  Define
	\begin{align}\label{Rp}
		R_{L^p}^p:=C^p\left(\frac{\kap^p}{\mu^p}F_{L^p}^p+\TLp^p\right)\quad\text{and}\quad \til{R}_{L^p}^p:=C^p\left(\frac{\kap^p}{\mu^p}\left(F_{L^p}^p+\left(\frac{M_{L^1}}{p}\right)^p\right)+\Tht_{L^p}^p+h^{2-p}M_{L^2}^p\right)
	\end{align}
where $C>0$ is an absolute constant.  We will show that Proposition \ref{prop:a priori0} implies the following $L^p$-bounds.

\begin{prop}\label{prop:lp}
Suppose $p>2$ and $\mu>0$.  There exist absolute constants $c_0, C_0>0$, independent of $p$, such that if \req{Rp} holds with $C=C_0$, and \req{l2:modes} holds, then
	\begin{align}\label{mp0:final}
		\Sob{\eta(t_2)}{L^p}^p\leq \left(\Sob{\eta(t_1)}{L^p}^p-p^p\frac{\mu^p}{\kap^p}\til{R}_{L^p}^p\right)e^{-c_0\kap(t_2-t_1)}+p^p\frac{{{{\mu^p}}}}{{{{\kap^p}}}}\til{R}_{L^p}^p,
	\end{align}
holds for all $0\leq t_1<t_2$.
In particular, we have
	\begin{align}\label{Mp}
		M_{L^p}(t_1,t_2)\leq \Sob{\eta(t_1)}{L^p}+C_0p\frac{\mu}{\kap}\til{R}_{L^p},
	\end{align}
where $M_{L^p}(t_1,t_2):=\sup_{t\in[t_1,t_2]}\Sob{\eta(t)}{L^p}$.
\end{prop}

To prove Proposition \ref{prop:lp}, we will make use of the following identity.

\begin{lem}\label{lem:avg}
Let $Q\sub\T^2$ open and $\phi\in L^2(Q)$.  Let  $\phi_Q:=\frac{1}{a({Q})}\int_{Q}\phi\ dx$, where $a(Q)$ denotes the area of $Q$.  Then
	\begin{align}\notag
		\Sob{\phi-\phi_{Q}}{L^2(Q)}^2=\Sob{\phi}{L^2(Q)}^2-a({Q})\phi_{Q}^2.
	\end{align}
\end{lem}
\begin{proof}
Simply observe that $(\phi-\phi_Q)^2=\phi^2-2\phi\phi_Q+\phi_Q^2$.  Thus
	\begin{align}\notag
		\Sob{\phi-\phi_Q}{L^2(Q)}^2=\int_Q\phi^2\ dx-2\phi_Q\int_Q\phi\ dx+a({Q})\phi_Q^2=\Sob{\phi}{L^2(Q)}^2-2a({Q})\phi_Q^2+a({Q})\phi_Q^2,
	\end{align}
as desired.
\end{proof}

\begin{proof}[Proof of Proposition \ref{prop:lp}]
Let $p\geq2$.  Upon multiplying \req{fb:ivp} by $\eta|\eta|^{p-2}$, integrating over $\T^2$, using the fact that $\int v\cdotp{\del}\eta\eta|\eta|^{p-2}\ dx=0$, then applying H\"older's inequality, Proposition \ref{lb}, and Young's inequality one arrives at
	\begin{align}\label{mp:1}
		\frac{1}p\frac{d}{dt}\Sob{\eta}{L^p}^p+\frac{2\kap}p\Sob{\Lam^{\gam/2}|\eta|^{p/2}}{L^2}^2&\leq \left(\Sob{f}{L^p}+{{{\mu}}}\Sob{J_h\eta}{L^p}+{{\mu}}\Sob{J_h\tht}{L^p}\right)\Sob{\eta}{L^p}^{p-1}\\
				&\leq \kap\frac{C^{p-1}(p-1)^{p-1}}{p}\frac{\mu^p}{\kap^p}\left( \frac{\kap^p}{\mu^p}F_{L^p}^p+\Sob{J_h\eta}{L^p}^p+\Sob{J_h\tht}{L^p}^p\right)+\frac{\kap}{2p}\Sob{\eta}{L^p}^p.\notag
	\end{align}

Observe that by \req{Ih:base0} we have $\Sob{J_h\tht}{L^p}\leq C\Sob{\tht}{L^p}\leq C\Tht_{L^p}$, which is finite by \req{tht:bounds}.  By \req{Ih:base1}, Proposition \ref{prop:a priori0}, and \req{eta:l2} we have
	\begin{align}\label{Ih:eta:lp}
		\Sob{J_h\eta}{L^p}\leq Ch^{2/p-1}\Sob{\eta}{L^2}\leq Ch^{2/p-1}M_{L^2}.
	\end{align}
From Corollary \ref{coro:poin} and Lemma \ref{lem:avg}, it follows that
	\begin{align}\label{poin:avg}
		\Sob{\eta}{L^p}^p-(4\pi^2)^{-1}\Sob{\eta}{L^{p/2}}^p=\Sob{|\eta|^{p/2}-(|\eta|^{p/2})_{\T^2}}{L^2}^2\leq C\Sob{\Lam^{\gam/2}|\eta|^{p/2}}{L^2}^2.
	\end{align}
By interpolation of $L^q$ spaces and Young's inequality, using the notation in \req{eta:lp} we have
	\begin{align}\label{p2:inter}
		\Sob{\eta}{L^{p/2}}^p\leq \Sob{\eta}{L^{p}}^{\frac{p(p-2)}{p-1}}M_{L^1}^{\frac{p}{p-1}}\leq C^p\left(\frac{p-2}{p-1}\right)^{p-2}M_{L^1}^{p}+\pi^2\Sob{\eta}{L^p}^p,
	\end{align}
for some absolute constant $C>0$.

Thus, upon collecting \req{Ih:eta:lp}, \req{poin:avg}, and \req{p2:inter}, we return to \req{mp:1}, and arrive at
	\begin{align}
		\frac{1}p\frac{d}{dt}\Sob{\eta}{L^p}^p+\frac{C\kap}{p}\Sob{\eta}{L^p}^p\leq C^p{\kap}\frac{(p-1)^{p-1}}{p}\frac{\mu^p}{\kap^p}\left(R_{L^p}^p+h^{2-p}M_{L^2}^p\right)+\kap\frac{C^p}{p}M_{L^1}^p.\notag
	\end{align}
It then follows from Gronwall's inequality 
and \req{Rp} that
	\begin{align}\label{gron:lp}
		\Sob{\eta(t_2)}{L^p}^p\leq \Sob{\eta(t_1)}{L^p}^pe^{-C\kap(t_2-t_1)}+p^p\frac{{{{\mu^p}}}}{{{{\kap^p}}}}\til{R}_{L^p}^p(1-e^{-C{{{\kap}}}(t_2-t_1)}),
	\end{align}
as desired.
\end{proof}

\subsection{$L^p$ to $H^\s$ uniform bounds}\label{sect:hs1}
Let $\mu>0$, $\Tht_{H^\s}$ be given by \req{sqg:hs:ball}, and $M_{L^r}$ by \req{Mp}.  Define
	\begin{align}\label{R:sig}
		\begin{split}
		&F_{H^{\s-\gam/2}}:=\frac{1}{{\kap}}\Sob{f}{H^{\s-\gam/2}},\quad R_{H^\s}^2:=C\left(\frac{{{{\kap}}}}{{{{\mu}}}}F_{H^{\s-\gam/2}}^2+\Tht_{H^\s}^2\right),\quad \Xi_{r,\al}(t_1,t_2):=C\left(\frac{M_{L^r}(t_1,t_2)}{\kap}\right)^{\frac{2\al}{\gam-1-2/r}},
		\end{split}
	\end{align}
where $C>0$ is an absolute constant.  For $h>0$, define
	\begin{align}\label{t:de}
		\rho(h,\s,\gam)&:=
				\begin{cases} h^{2\s},&\s\leq\gam/2,\\
							h^\gam, &\s>\gam/2.
\end{cases}
	\end{align}

\begin{prop}\label{prop:hs:a priori}
There exist absolute constants $C_0, c_0>0$, with $c_0$ depending on $C_0$, such that  if \req{R:sig} holds, with $C=C_0$ and $h>0$ satisfies
	\begin{align}\label{tj:modes}
		\frac{\mu}{\kap}\rho(h,\s,\gam)\leq c_0,
	\end{align}
then when $\s\leq\gam/2$, we have that
	\begin{align}\label{hs:ineq}
		\Sob{\eta(t_2)}{H^\s}^2\leq \left[\Sob{\eta(t_1)}{H^\s}^2+\Xi_{p,\s}(t_1,t_2)\left(\Sob{\eta(t_1)}{L^2}^2-R_{L^2}^2\right)-R_{H^\s}^2\right]e^{-{{{\mu}}}(t_2-t_1)}+\Xi_{p,\s}(t_1,t_2)R_{L^2}^2+R_{H^\s}^2,\quad
	\end{align}
holds for all $0\leq t_1<t_2$ and
	\begin{align}\label{hsgam:ineq}
		{{\kap}}\int_0^t\Sob{\eta(s)}{H^{\s+\gam/2}}^2\ ds\leq \Sob{\eta_0}{H^\s}^2+\Xi_{p,\s}(0,t)\Sob{\eta_0}{L^2}^2+{\mu}t\left(R_{H^\s}^2+\Xi_{p,\s}(0,t)R_{L^2}^2\right),\quad  t>0.
	\end{align}
In particular, we have
	\begin{align}\label{Ms}
		\Sob{\eta(t)}{H^\s}^2\leq \Sob{\eta_0}{H^\s}^2+R_{H^\s}^2+\Xi_{p,\s}M_{L^2}^2:=M_{H^\s}^2,\quad t>0,
	\end{align}
where
	\begin{align}\label{xi:infty}
		\Xi_{r,\al}:=\Xi_{r,\al}(0,\infty).
	\end{align}

On the other hand, when $\s>\gam/2$, then there exist absolute constants $C_0, c_0>0$ such that if \req{R:sig} holds with $C=C_0$ and $h>0$ satisfies \req{tj:modes}, then
	\begin{align}\label{prop:a priori00}
		\Sob{\eta(t)}{H^\s}^2\leq\Sob{\eta_0}{H^\s}^2+C_0\left(F_{H^{\s-\gam/2}}^2+\frac{\mu}{h^{2\s}\kap}\left(M_{L^2}^2+\Tht_{L^2}^2\right)\right)+\Xi_{p,\s}M_{H^{\gam/2}}^2,\quad t>0.
	\end{align}
\end{prop}

\begin{proof}
Let $t\in[t_1,t_2]$.  Multiply \req{fb:ivp} by $\Lam^{2\s}\eta$ and integrate over $\T^2$ to obtain
	\begin{align}\label{hs:initial}
		\frac{1}2\frac{d}{dt}\Sob{\Lam^{\s}\eta}{L^2}^2&+{{{\kap}}}\Sob{\Lam^{\s+\gam/2}\eta}{L^2}^2+{{{\mu}}}\Sob{\Lam^{\s}\eta}{L^2}^2\notag\\
					&=-\int v\cdotp{\del}\eta\Lam^{2\s}\eta\ dx+\int f\Lam^{2\s}\eta\ dx+{{{\mu}}}\int (\eta-J_h\eta)\Lam^{2\s}\eta\ dx+{{{\mu}}}\int J_h\tht\Lam^{2\s}\eta\ dx\notag\\
					&=I + II + III + IV.
	\end{align}
We estimate $I, II$ first.  We will estimate $III, IV$ depending on whether $J_h$ is of Type I or II.

To estimate $I$, we first use the facts that $\del\cdotp(v\eta)=v\cdotp{\del}\eta$ and $\Ri\sim\del\Lam^{-1}$, i.e., operator with symbol $-i\xi/\abs{\xi}$, to rewrite $I$ with Parseval's identity as
	\begin{align}\label{setup1}
		I=\int \Lam^{\s-\gam/2}\del\cdotp(v\eta)\Lam^{\s+\gam/2}\eta\ dx=\int\Lam^{1+\s-\gam/2}\Ri\cdotp(v\eta)\Lam^{\s+\gam/2}\eta\ dx.
	\end{align}

Now, since $1<p<\infty$ satisfies $1-\s<2/p$, Sobolev embedding ensures $H^\s\imb L^p$.  Let $q>1$ be such that $1/p+1/q=1/2$.  By Proposition \ref{prod:rule}, the fact that $\Ri, \Ri^\perp$ commute with $\Lam$ and are Calder\'on-Zygmund operators, i.e., $\Sob{v}{L^r}\leq C\Sob{\eta}{L^r}$, for all $r\in(1,\infty)$, and since $H^{2/p}\imb L^q$, we may estimate $I$ as
	\begin{align}\label{est:I1}
		|I|&\leq \Sob{\Lam^{1+\s-\gam/2}(v\eta)}{L^2}\Sob{\Lam^{\s+\gam/2}\eta}{L^2}\notag\\
			&\leq  C\Sob{\Lam^{1+\s-\gam/2}\eta}{L^q}\Sob{\eta}{L^p}\Sob{\Lam^{\s+\gam/2}\eta}{L^2}\notag\\
			&\leq C\Sob{\Lam^{1+\s-\gam/2+2/p}\eta}{L^2}\Sob{\eta}{L^{p}}\Sob{\Lam^{\s+\gam/2}\eta}{L^2}.
	\end{align}
By interpolation (Proposition \ref{interpol}), we have
	\begin{align}\label{hs:interpolate}
		\Sob{\Lam^{1+\s-\gam/2+2/p}\eta}{L^2}\leq \Sob{\Lam^{\s+\gam/2}\eta}{L^2}^{\frac{\s-(\gam-1-2/p)}{\s}}\Sob{\Lam^{\gam/2}\eta}{L^2}^{\frac{\gam-1-2/p}\s}.
	\end{align}
Therefore, returning to \req{est:I1}, from \req{hs:interpolate} and Young's inequality,  we have
	\begin{align}
		|I|&\leq C\Sob{\Lam^{\s+\gam/2}\eta}{L^2}^{\frac{2\s-(\gam-1-2/p)}{\s}}\Sob{\Lam^{\gam/2}\eta}{L^2}^{\frac{\gam-1-2/p}\s}\Sob{\eta}{L^{p}}\notag\\
		&\leq \frac{{{{\kap}}}}8\Sob{\Lam^{\s+\gam/2}\eta}{L^2}^2+\frac{1}2\Xi_{p,\s}(t_1,t_2){{{\kap}}}\Sob{\Lam^{\gam/2}\eta}{L^2}^2,\quad t\in[t_1,t_2],\notag
	\end{align}
where $\Xi_{p,\s}(t_1,t_2)$ is given in \req{R:sig}.  Note that this quantity is finite due from Proposition \ref{prop:lp}.

For $II$, we apply the Plancherel relation, Cauchy-Schwarz inequality, and Young's inequality, to obtain
	\begin{align}\notag
		\begin{split}
		|II|&\leq \frac{1}{{{{\kap}}}}\Sob{\Lam^{\s-\gam/2}f}{L^2}^2+\frac{{{{\kap}}}}{4}\Sob{\Lam^{\s+\gam/2}\eta}{L^2}^2.
		\end{split}
	\end{align}

Now we estimate $III$ and $IV$.  We split the treatment of these terms into two cases: $\s\leq\gam/2$ and $\s>\gam/2$.  Note that the estimates hold for both Type I and II operators, although we will only make use of Type I properties.

\subsubsection*{Case: $\s\leq\gam/2$}
We estimate $III$ by applying the Cauchy-Schwarz inequality, the Poincar\'e inequality, \req{TI}, \req{t:de}, \req{tj:modes} (with $c_0$ sufficiently small), then applying Young's inequality we estimate
	\begin{align}\label{t1:III}
		|III|&\leq \mu\Sob{\eta-J_h\eta}{L^2}\Sob{\Lam^{2\s}\eta}{L^2}\leq C\mu h^\s\Sob{\Lam^\s\eta}{L^2}\Sob{\Lam^{\s+\gam/2}\eta}{L^2}\notag\\
		&\leq \frac{\kap}{16}\Sob{\Lam^{\s+\gam/2}\eta}{L^2}^2+\frac{\mu}{4}\Sob{\Lam^\s\eta}{L^2}^2.
	\end{align}
Note that we used the fact that $(H1)$ implies $\s<1$.  On the other hand, using \req{TI}, the Cauchy-Schwarz and Poincar\'e inequalities, \req{t:de},  and \req{tj:modes}, we similarly estimate $IV$ as
	\begin{align}
		|IV|&\leq\mu\Sob{\tht-J_h\tht}{L^2}\Sob{\Lam^{2\s}\eta}{L^2}+\mu\Sob{\Lam^\s\tht}{L^2}\Sob{\Lam^\s\eta}{L^2}\notag\\
			&\leq C\mu\frac{\mu h^{2\s}}{\kap}\Sob{\tht}{H^\s}^2+\frac{\kap}{16}\Sob{\Lam^{\s+\gam/2}\eta}{L^2}^2+C{\mu}\Sob{\Lam^\s\tht}{L^2}^2+\frac{\mu}4\Sob{\Lam^\s\eta}{L^2}^2\notag\\
			&\leq C\mu\Tht_\s^2+\frac{\kap}{16}\Sob{\Lam^{\s+\gam/2}\eta}{L^2}^2+\frac{\mu}4\Sob{\Lam^\s\eta}{L^2}^2.\notag
	\end{align}
Upon combining $I-IV$, we arrive at
	\begin{align}\label{hs:noavg1}
		\frac{d}{dt}\Sob{\eta}{H^\s}^2+{{{\mu}}} \Sob{\eta}{H^\s}^2+{{{{\kap}}}}\Sob{\eta}{H^{\s+\gam/2}}^2&\leq 2{{{\kap}}}F_{H^{\s-\gam/2}}^2+2{{{\mu}}} C\Tht_{H^\s}^2+\Xi_{p,\s}(t_1,t_2){{{\kap}}}\Sob{\eta}{H^{\gam/2}}^2.
	\end{align}
Thus, from Gronwall's inequality applied over $[t_1, t_2]$ we obtain
	\begin{align}\notag
		\Sob{\eta(t_2)}{H^\s}^2\leq \Sob{\eta(t_1)}{H^\s}^2e^{-{{{\mu}}}(t_2-t_1)}+2R_{H^\s}^2(1-e^{-{{{\mu}}}(t_2-t_1)})+\Xi_{p,\s}(t_1,t_2)\left({{{{\kap}}}}\int_{t_1}^{t_2}e^{-{{{\mu}}}(t_2-s)}\Sob{\eta(s)}{H^{\gam/2}}^2\ ds\right).
	\end{align}
We may then apply Proposition \ref{prop:a priori0} to bound the last term above and obtain \req{hs:ineq}.

On the other hand, letting $t_1=0$ and $t_2=t$, then integrating \req{hs:noavg1} over $[0,t]$ gives
	\begin{align}\notag
		{{\kap}}\int_0^t\Sob{\eta(s)}{H^{\s+\gam/2}}^2\ ds\leq\Sob{\eta_0}{H^\s}^2+2t{\mu}R_{H^\s}^2+\Xi_{p,\s}(0,t)\left({{\kap}}\int_0^t\Sob{\eta(s)}{H^{\gam/2}}^2\ ds\right),
	\end{align}
and we again use Proposition \ref{prop:a priori0} to bound the last term and obtain \req{hsgam:ineq}.  This establishes the case $\s\leq\gam/2$.

\subsubsection*{Case: $\s>\gam/2$}
Since $\Sob{\eta_0}{H^{\gam/2}}\lesssim\Sob{\eta_0}{H^\s}$, it follows by estimating exactly as above that $\eta$ is uniformly bounded in $H^{\gam/2}$, provided that \req{tj:modes} holds with $\rho=h^{\gam}$. To obtain uniform estimates in $H^\s$, we do not insert a damping term, so that we replace \req{hs:initial} by
	\begin{align}\label{hs:initial2}
		\frac{1}2\frac{d}{dt}\Sob{\Lam^{\s}\eta}{L^2}^2+{{{\kap}}}\Sob{\Lam^{\s+\gam/2}\eta}{L^2}^2=I + II + III' + IV,
	\end{align}
where $I, II, IV$ are as before and $III'$ is given by
	\begin{align}\notag
		III'=-\mu\int J_h\eta \Lam^{2\s}\eta\ dx.
	\end{align}
We treat $I, II$ the same.  To estimate $III', IV$, we apply Plancherel's theorem, the Cauchy-Schwarz inequality, \req{Ih:base2}, the Poincar\'e inequality, and Young's inequality to obtain
	\begin{align}\notag
		|III'|&\leq\mu\Sob{J_h\eta}{H^{\s-\gam/2}}\Sob{\eta}{H^{\s+\gam/2}}\notag\\
			&\leq C\mu h^{-(\s-\gam/2)}\Sob{\eta}{L^2}\Sob{\eta}{H^{\s+\gam/2}}\notag\\
			&\leq C\frac{\mu^2}{h^{2\s-\gam}\kap}\Sob{\eta}{L^2}^2+\frac{\kap}{8}\Sob{\eta}{H^{\s+\gam/2}}^2\notag\\
		|IV|&\leq C\frac{\mu}{h^{2\s}}\Sob{\tht}{L^2}^2+\frac{\mu}{2}\Sob{\eta}{H^{\s}}^2.\notag	\end{align}

Then combining $I-IV$, we apply the Poincar\'e's inequality to arrive at
	\begin{align}\notag
		\frac{d}{dt}\Sob{\eta}{H^\s}^2+{{{{\kap}}}}\Sob{\eta}{H^{\s}}^2&\leq 2{{{\kap}}}F_{\s,\gam}^2+\Xi_{p,\s}(t_1,t_2){{{\kap}}}\Sob{\eta}{H^{\gam/2}}^2+C\frac{\mu}{h^{2\s}}\left(\frac{\mu h^\gam}{\kap}\Sob{\eta}{L^2}^2+\Tht_{L^2}^2\right).
	\end{align}
Then by \req{tj:modes}, Gronwall's inequality, Propositions \ref{prop:a priori0} and \ref{prop:lp}, we obtain \req{prop:a priori00}.
\end{proof}

\begin{rmk}\label{hs:modes:case}
We observe that in the case that $J_h$ is Type II, we need not treat the cases $\s\leq\gam/2$ and $\s>\gam/2$ separately.  Indeed, for any $\s>2-\gam$, observe that by the Cauchy-Schwarz inequality, \req{TII}, Young's inequality, \req{t:de}, \req{tj:modes}, and $c_0$ sufficiently small we may estimate
	\begin{align}\label{t2:III}
		|III|&\leq \mu\Sob{\Lam^\s(\eta-J_h\eta)}{L^2}\Sob{\Lam^\s\eta}{L^2}\notag\\
			&\leq\frac{{{{\kap}}}}8\Sob{\Lam^{\s+\gam/2}\eta}{L^2}^2+C\frac{h^\gam{{{\mu}}}^2}{{{{\kap}}}} \Sob{\Lam^{\s}\eta}{L^2}^2\leq\frac{{{{\kap}}}}8\Sob{\Lam^{\s+\gam/2}\eta}{L^2}^2+\frac{\mu}4\Sob{\Lam^{\s}\eta}{L^2}^2.
	\end{align}
Using \req{Ih:base0}, \req{TII}, and since $\tht_0\in\mathcal{B}_{H^\s}$, we estimate $IV$ as
	\begin{align}
		|IV|&\leq \mu\Sob{J_h\Lam^\s\tht}{L^2}\Sob{\Lam^\s\eta}{L^2}\leq C{{{\mu}}}\Tht_{H^\s}^2+\frac{{{{\mu}}}}{4}\Sob{\Lam^{\s}\eta}{L^2}^2.\notag
	\end{align}
Thus, we may combine these estimates with those for $I,II$ and apply Gronwall's inequality to arrive at \req{hsgam:ineq}.
\end{rmk}

\section{Proofs of Main Theorems}\label{sect:pf1}

For the proofs of both Theorems \ref{thm1} and \ref{synch}, we assume the Standing Hypotheses, $(H1)-(H8)$.  Let $F_{L^2}, F_{H^{-\gam/2}}$, $F_{L^p}$, $F_{\s-\gam/2}$, $\Tht_{L^2}, \Tht_{L^p}, \Tht_{H^\s}$, and $R_{L^2}, R_{L^2}'$, $R_{L^p}, R_{L^p}'$, $R_{H^\s}, \Xi_{r,\al}$ be defined as in Propositions \ref{prop:sqg:ball}, \ref{prop:a priori0}, \ref{prop:lp}, \ref{prop:hs:a priori}, and Proposition \ref{prop:ga}.
Let $C_0'$ be the maximum among the constants, $C_0$, appearing in Propositions \ref{prop:a priori0}, \ref{prop:lp}, \ref{prop:hs:a priori}, and let $c_0'$ be the minimum among the constants, $c_0$, appearing there.
We assume that $c_0'\ll1$.   Suppose \req{num:modes} holds with $\rho$ given by \req{t:de}.
This implies that $\mu, h$ satisfy
	\begin{align}\label{main:mu}
		\frac{\mu}{\kap}\max\{\rho(h,\s,\gam),h^\gam\}\leq c_0'.
	\end{align}

We first prove Theorem \ref{thm1}, i.e., the short-time existence of strong solutions and establish uniqueness within this class of solutions.  We will then prove Theorem \ref{synch}, which establishes the synchronization property of the algorithm.

\subsection{Proof of Theorem \ref{thm1}}

\subsubsection*{Existence of strong solutions}
If $\eta_0, \tht_0\in L_{per}^2(\T^2)\cap\mathcal{Z}$ and $f\in V_{-\gam/2}$, then the assumption \req{main:mu}, and the a priori bounds of Propositions \ref{prop:sqg:ball}, \ref{prop:a priori0} guarantee the existence of weak solutions (cf. \cite{resnick}).  To show that strong solutions exist, i.e., the solutions to \req{fb:ivp} belong to $V_\s$, $\s>2-\gam$, provided that $\eta_0\in V_\s$, $\tht_0\in\mathcal{B}_\s$, and $f\in V_{\s-\gam/2}\cap L_{per}^p(\T^2)$, where $1-\s<2/p<\gam-1$, we need only establish a priori estimates.  Indeed, by adding an artificial viscosity $-\nu\De$ to \req{fb:ivp} and mollifying $f$ by $f^{(\nu)}$, we have global smooth solutions $\eta^{(\nu)}$ (\textit{independent} of $\gam$) such that the weak limit $\eta^{(\nu)}\goesto\eta$, as $\nu\goesto0^+$, is a weak solution of \req{fb:ivp} (cf. \cite{const:wu:qgwksol, ctv}).  Since \req{main:mu} holds, the family satisfies precisely the same estimates performed above in establishing Propositions \ref{prop:a priori0}, \ref{prop:lp}, and \ref{prop:hs:a priori}.
  Consequently, these bounds are inherited in the limit, thus ensuring that $\eta$ is a strong solution.

\subsubsection*{Uniqueness of strong solutions}
Let $\eta_1, \eta_2$ be strong solutions of \req{fb:ivp} with initial data $\eta_{0}^{(1)}, \eta_{0}^{(2)}\in V_\s$, $\s>2-\gam$, respectively, and $\tht$ the strong solution of \req{sqg} with initial data $\tht_0\in \mathcal{B}_\s$.  Let $\z:=\eta_1-\eta_2$.  Then the evolution of $\z$ is given by
	\begin{align}\label{z:eqn}
		\begin{cases}
			&\bdy_t\z+{{{\kap}}}\Lam^\gam\z+(\Ri^\perp\z)\cdotp{\del}\z+(\Ri^\perp\z)\cdotp{\del} \eta_2+v_2\cdotp{\del}\z=-{{{\mu}}} J_h\z,\quad x\in \T^2, t>0,\\
			&\z(x,0)=\eta^{(1)}_0(x)-\eta^{(2)}_0(x),\quad v_1=\Ri^\perp \eta_1, \quad v_2=\Ri^\perp\eta_2.
		\end{cases}
	\end{align}
We multiply by $\psi:=-\Lam^{-1}\z$ and integrate to obtain:
	\begin{align}\notag
		\frac{1}2\frac{d}{dt}\Sob{\psi}{{H}^{1/2}}^2+&{{{\kap}}}\Sob{\psi}{H^{(\gam+1)/2}}^2+{{{\mu}}}\Sob{\psi}{{H}^{1/2}}^2=\underbrace{\int v_2\cdotp{\del}\z\psi\ dx}_I+\underbrace{{{{\mu}}} \int (\z-J_h\z)\psi\ dx}_{II}.
	\end{align}
Note that we have used the orthogonality property $\int (\Ri^\perp f)\cdotp \del g\Lam^{-1}f\ dx=0$.

We estimate $I$ as follows.  First observe that by integrating by parts and using the relation $\psi=-\Lam^{-1}\z$, we may write
	\begin{align}\notag
		I=\int v_2\cdotp{\del}\z\psi\ dx=\int (v_2\Lam\psi)\cdotp{\del}\psi\ dx.
	\end{align}
Thus, (as in p. 32 of \cite{resnick}), we may apply H\"older's inequality, the Calder\`on-Zygmund theorem, and the Sobolev embedding $H^{1/p}\imb L^q$ to obtain
	\begin{align}\label{z:nlt:1}
		\left|I\right|\leq Cp\Sob{\eta_2}{L^p}\Sob{\z}{L^{q}}\Sob{{\del}\psi}{L^q}\leq Cp\Sob{\eta_2}{L^p}\Sob{\psi}{H^{1+1/p}}^2,
	\end{align}
where $1/p+2/q=1$. Since $p>2/(\gam-1)$ from $(H3)$, by interpolation we have
	\begin{align}\notag
		\Sob{\psi}{H^{1+1/p}}\leq C\Sob{\psi}{H^{(\gam+1)/2}}^{\frac{1+2/p}{\gam}}\Sob{\psi}{{H}^{1/2}}^{\frac{\gam-1-2/p}{\gam}}.
	\end{align}
It follows that
	\begin{align}\notag
		\left| I\right|\leq \frac{{{{\kap}}}}{4}\Sob{\psi}{H^{(\gam+1)/2}}^2+\frac{{\kap}}2p^{\frac{\gam}{\gam-1-2/p}}\Xi_{p,\gam/2}{}\Sob{\psi}{H^{1/2}}^2.
	\end{align}
where $\Xi_{p,\gam/2}$ is defined by \req{R:sig} and \req{xi:infty}, except in terms of $\Sob{\eta_{2}}{L^p}$ and with $C=C_0'$ there.  Note that we will suppress the dependence of the constant on $p$.  Also note that $\Xi_{p,\gam/2}<\infty$ is guaranteed by Proposition \ref{prop:lp} since ${{\mu}}$ satisfies \req{main:mu}.

The estimate of $II$ is, in fact, independent of the Type of $J_h$.  Indeed, if $J_h$ is Type II, then $J_h$ also satisfies \req{TI}.

So observe that by \req{TI}, the Cauchy-Schwarz and Young's inequalities, \req{main:mu}, and by interpolation (Proposition \ref{interpol}) we have
	\begin{align}\notag
		|II|&\leq\mu\Sob{\z-J_h\z}{H^{-\gam/2}}\Sob{\psi}{H^{\gam/2}}\notag\\
			&\leq C\mu h^{\gam/2}\Sob{\psi}{H^{1}}\Sob{\psi}{H^{\gam/2}}\notag\\
			&\leq C\mu h^{\gam/2}\Sob{\psi}{H^{(\gam+1)/2}}\Sob{\psi}{H^{1/2}}\notag\\
			&\leq C\frac{\mu^2h^\gam}{\kap}\Sob{\psi}{H^{1/2}}^2+\frac{\kap}4\Sob{\psi}{H^{(\gam+1)/2}}^2.\notag
	\end{align}

Thus, upon combining $I, II$, and \req{num:modes}, we may deduce
	\begin{align}\label{uniq:1}
		\frac{d}{dt}\Sob{\psi}{{H}^{1/2}}^2+{{{\kap}}}\Sob{\psi}{H^{(\gam+1)/2}}^2+\frac{3}2{{{\mu}}}\Sob{\psi}{{H}^{1/2}}^2\leq{{{\kap}}} {\Xi}_{p,\gam/2}\Sob{\psi}{{H}^{1/2}}^2.
	\end{align}
By Gronwall's inequality, we have
	\begin{align}
		\Sob{\psi(t)}{H^{1/2}}^2\leq \Sob{\psi_0}{H^{1/2}}^2e^{(\kap\Xi_{p,\gam/2}-(3/2)\mu)t}.
	\end{align}
which establishes continuous dependence on initial conditions for \req{fb:ivp} in the topology $H^{-1/2}$ (since $\psi=-\Lam^{-1}\z$).  In fact, by interpolation (see \req{interpolate:1}), we may establish continuous dependence in $H^{\s-\eps}$ for all $\eps\in(0,\s+1/2)$.  In particular, if $\eta_0^{(1)}=\eta_0^{(2)}$, then $\z_0\equiv0$ and \req{uniq:1} implies $\z\equiv0$, which establishes uniqueness of solutions.

This completes the proof of Theorem \ref{thm1}.\qed

\subsection{Proof of Theorem \ref{synch}}\label{sect:synch:1}
Let $\eta$ be the unique global strong solution of \req{fb:ivp} with initial data $\eta_0\in V_\s$, $\s>2-\gam$, and $\tht$ be the unique global strong solution of \req{sqg} with initial data $\tht_0\in \mathcal{B}_\s$.  We assume \req{mu:assumpt}, which is
	\begin{align}\label{synch:cond}
		 \mu\geq c_0'{\kap}\left(\frac{{\TLp}}{\kap}\right)^{\gam/(\gam-1-2/p)},
	\end{align}
where $c_0'$ is the constant from \req{main:mu} (whose magnitude is determined below).  Let $\z:=\eta-\tht$ and $\psi:=-\Lam^{-1}\z$.  Then, proceeding as in the proof of uniqueness from the previous section, we arrive at
	\begin{align}\label{synch:bal1}
		\frac{d}{dt}\Sob{\psi}{{H}^{1/2}}^2+{{{\mu}}}\left(\frac{3}2-{\Xi}_{p,\gam/2}(\tht)\frac{{{{\kap}}}}{{{{\mu}}}}\right)\Sob{\psi}{{H}^{1/2}}^2\leq0,
	\end{align}
where $\Xi_{p,\gam/2}(\tht)$ is defined by \req{R:sig} and \req{xi:infty}, except with $\Tht_{L^p}$ replacing $M_{L^p}$, i.e.,
	\begin{align}\label{synch:xi}
		\Xi_{p,\gam/2}(\tht):=C\left(\frac{\TLp}{\kap}\right)^{\frac{\gam}{\gam-1-2/p}},
	\end{align}
for some absolute constant $C>0$.
Note that $\Xi_{p,\gam/2}(\tht)$ is \textit{independent} of $h,{{\mu}}$.  Therefore, by \req{synch:cond} with $c_0'$ \textit{chosen sufficiently small}, it follows that
	\begin{align}\notag
		\Sob{\psi}{H^{1/2}}^2\leq\Sob{\psi_0}{H^{1/2}}^2e^{-{{{{\mu}}}} t}.
	\end{align}

To upgrade the convergence, we need only interpolate since Proposition \ref{prop:hs:a priori} ensures that $\eta$ satisfies uniform bounds in $H^\s$.  Indeed, let
	\begin{align}\notag
		M_{H^\s}:=\sup_{t\geq0}\Sob{\eta(t)}{H^\s}.
	\end{align}
 Observe that for any $\s>2-\gam$ and $0<\eps<\s+1/2$, we have
	\begin{align}\label{interpolate:1}
		\Sob{\psi(t)}{H^{\s+1-\eps}}\leq C\Sob{\psi(t)}{H^{\s+1}}^{\frac{\s+1/2-\eps}{\s+1/2}}\Sob{\psi(t)}{H^{1/2}}^{\frac{\eps}{\s+1/2}}.
	\end{align}
Thus
	\begin{align}\notag
		\Sob{\eta(t)-\tht(t)}{H^{\s-\eps}}^2\leq {O_\eps(e^{-\mu_\eps t})},
	\end{align}
where $\mu_\eps=\eps{{{\mu}}}/(\s+1/2)$, and
	\begin{align}\label{O:eps1}
		O_\eps(e^{-\mu_\eps t}):=C\left(M_{H^\s}+\Tht_{H^\s}\right)^{\frac{2(\s+1/2-\eps)}{\s+1/2}}\Sob{\psi_0}{H^{1/2}}^{\frac{2\eps}{\s+1/2}}e^{-{\mu_\eps} t}.
	\end{align}
In particular, this holds for $\eps=\s$, so that
	\begin{align}\label{l2:sync}
		\Sob{\eta(t)-\tht(t)}{L^2}^2\leq O_\s(e^{-c_\s{{\mu}}t}),
	\end{align}
where $c_\s=\s/(\s+1/2)$.  This establishes \req{synch:high1} upon rescaling,  concluding the proof of Theorem \ref{synch}.\qed

\begin{rmk}\label{rmk:logmu}
Note that even though $M_{H^\s}$ depends on $h, {{\mu}}$, it will only affect the exponential rate, $c_\s{{\mu}}$, up to a fixed, multiplicative factor.  Thus, the synchronization still occurs at an exponential rate.
\end{rmk}

\appendix

\section{}\label{app}

In this section, we verify that volume element (see \req{vol:elts}, \req{vol:elts:shift} below) and modal projection interpolants (see \req{rough:proj}, \req{smooth:proj}) are of Type I and II, respectively.  For convenience, we let $\Om=\T^2$ denote the $2\pi$-periodic box throughout, where $\T^2=(\R/(2\pi\Z))^2$, so that $\T^2=[-\pi,\pi]^2$.  Let $\mathcal{Z}$ and $V_\be$ be the spaces defined by \req{Z:space} and \req{hper}, respectively.  The main claim is the following:

\begin{prop}\label{thm:Jh}
Suppose that $J_h$ is a linear operator that is defined by either \req{vol:elts}, \req{vol:elts:shift}, \req{modes}, \req{smooth:proj}, below.  Then
	\begin{enumerate}[(Property 0.1)]
	\item $\sup_{h>0}\Sob{J_h\phi}{L^p}\ls\Sob{\phi}{L^p}$, for any $p\in(1,\infty)$, $\phi\in {L}_{per}^p(\T^2)$;
	\item $\Sob{J_h\phi}{L^p}\ls h^{2/p-1}\Sob{\phi}{L^2}$, for any $p\in[2,\infty)$, $\phi\in {L}_{per}^2(\T^2)$;
	\item $\Sob{J_h\phi}{\dot{H}^\be}\ls h^{-\be}\Sob{\phi}{L^2}$, for any $\be\geq0$, $\phi\in {L}_{per}^2(\T^2)$.
	\end{enumerate}
If $J_h$ is defined by \req{vol:elts}, \req{vol:elts:shift}, then for any $\be\in(0,1]$ we have
	\begin{enumerate}[(Property 1.1)]
		\item $\Sob{\phi-{J}_h\phi}{L^2}\lesssim h^{\be}\Sob{\phi}{\dot{H}^\be}$, for any $\phi\in\dot{H}_{per}^\be(\T^2)$\ ($V_\be$ if $J_h$ is \req{vol:elts:shift});
		\item[(Property 1.2)] $\Sob{\phi-{J}_h\phi}{\dot{H}^{-\be}}\lesssim h^{\be}\Sob{\phi}{L^2}$, for any $\phi\in{L}_{per}^2(\T^2)$\ ($L^2_{per}(\T^2)\cap\mathcal{Z}$ if $J_h$ is \req{vol:elts:shift}).
	\end{enumerate}
 If $J_h$ is defined by \req{rough:proj} or \req{smooth:proj}, then for any $\phi\in\dot{H}^\be_{per}(\T^2)$ we have
	\begin{enumerate}[(Property 2.1)]
		\item $\Sob{\phi-J_h\phi}{\dot{H}^\al}\ls h^{\be-\al}\Sob{\phi}{\dot{H}^\be}$,  for any $\phi\in\dot{H}^\be_{per}(\T^2), \be>\al$;
		\item $\Lam^\be J_h\phi=J_h\Lam^\be\phi$, for any $\phi\in\dot{H}^\be_{per}(\T^2), \be\in\R$.
	\end{enumerate}
Moreover, when $J_h$ is given by \req{vol:elts}, \req{vol:elts:shift}, or \req{smooth:proj}, then (Property 0.1) and (Property 0.2) are valid for $p=1,\infty$ and $p=\infty$, respectively.
\end{prop}

We then define Type I operators as any linear operator, $J_h$, that satisfy Property (0.1)-(0.3) \text{and} Property (1.1), (1.2), while Type II operators are those that satisfy Property (0.1)-(0.3) \text{and} Property (1.1), (1.2).

\subsection{Local averages (Type I)}
Let us recall the following construction of a partition of unity from \cite{azouani:olson:titi}.  Let $N>0$ be a perfect square integer and partition $\Om$ into $4N$ squares of side-length $h=\pi/\sqrt{N}$.  Let $\mathcal{J}=\{0,\pm1,\pm2,\dots,\pm(\sqrt{N}-1), -\sqrt{N}\}^2$ and for each $\al\in\mathcal{J}$, define the semi-open square
	\begin{align}
		Q_\al=[ih,(i+1)h)\times[jh,(j+1)h),\quad\text{where}\quad \al=(i,j)\in\mathcal{J}.\notag
	\end{align}
Let $\mathcal{Q}$ denote the collection of all $Q_\al$, i.e.
	\begin{align}
		\mathcal{Q}:=\{Q_\al\}_{\al\in\mathcal{J}}.\notag
	\end{align}
Consider the functions
	\begin{align}
		\chi_{\al}(x):=\ind_{Q_\al}(x)\quad\text{and}\quad\psi_\al(x):=\sum_{k\in\Z^2}\chi_\al(x+2\pi k).\notag
	\end{align}
In particular, $\psi_\al$ is the $2\pi$-periodic extension of the characteristic function, $\ind_{Q_\al}$, of $Q_\al$ to $\R^2$.

Given $\eps>0$ fixed, we mollify $\psi_\al$ as
	\begin{align}
		\til{\psi}_\al(x)=(\rho_\eps*\psi_\al)(x),\quad x\in\T^2,\notag
	\end{align}
with the function $\rho_\eps(x):=\eps^{-2}\rho(x/\eps)$, where $\rho$ is given by
	\begin{align}
		\rho(x):=\begin{cases} K_0\exp\left(\frac{-1}{1-\abs{x}^2}\right),&\abs{x}<1\\
							0,&\abs{x}\geq1,
					\end{cases}\notag
	\end{align}
and $K_0>0$ is the absolute constant given by
	\begin{align}
		K_0^{-1}=\int_{\abs{x}<1}\exp\left(\frac{-1}{1-\abs{x}^2}\right)\ dx.\notag
	\end{align}
	
Now suppose that $N\geq9$ and $\eps=h/10$.  For each $\al=(i,j)\in\mathcal{J}$, let us also define the augmented squares, $\hat{Q}_\al$ and $Q_\al(\eps)$, by
	\begin{align}
		\hat{Q}_\al:=[(i-1)h, (i+2)h]\times[(j-1)h,(j+2)h]\quad\text{and}\quad Q_\al(\eps):=Q_\al+B(0,\eps).\notag
	\end{align}
so that $Q_\al\sub Q_\al(\eps)\sub\hat{Q}_\al$ for each $\al\in\mathcal{J}$, and the ``core," $C_\al(\eps)$, by
	\begin{align}
		C_\al(\eps):=Q_\al(\eps)\smod\bigcup_{\al'\neq\al}Q_{\al'}(\eps)\neq\varnothing,\quad \al\in\mathcal{J}.\notag
	\end{align}

Finally, for $\phi\in L^1(\T^2)$, define
	\begin{align}\label{glob:avg}
		\lb\phi\rb:=\frac{1}{4\pi^2}\int_{\T^2}\phi(x)\ dx.
	\end{align}
Then we have the following proposition, which follows from the definition of $\til{\psi}_\al$.  We note that properties $(i)-(iii)$ below can be found in \cite{azouani:olson:titi}, while property $(iv)$ can be proved by using a characterization of the Sobolev space norm (see Remark \ref{rmk:sob:char} below and rescaling Proposition \ref{lem:ben:oh} (below) and using the fact that $\{\psi_\al\}$ satisfies $(i)-(iii)$ (see Remark \ref{rmk:sob:char} below and the proof of Corollary \ref{coro:poin} for relevant details).

\begin{prop}\label{prop:pou}
Let $N\geq9$, $h:=L/\sqrt{N}$, and $\eps:=h/10$.  The collection $\{\psi_{\al}\}_{\al\in\mathcal{J}}$ forms a smooth partition of unity satisfying
	\begin{enumerate}[(i)]
		\item $0\leq\til{\psi}_\al\leq1$ and $\spt\til{\psi}_{\al}\sub ({Q}_\al(\eps)+(2\pi\Z)^2)$;
		\item $\til{\psi}_\al=1$, for all $x\in (C_\al(\eps)+(2\pi\Z)^2)$ and $\sum_{\al\in\mathcal{J}}\til{\psi}_\al(x)=1$, for all $x\in\R^2$;
		\item $\lb\til{\psi}_\al\rb=(h/(2\pi))^2$ and $4h/5\leq\Sob{\til{\psi}_\al}{L^2(\Om)}\leq6h/5$.
		\item $\sup_{\al\in\mathcal{J}}\Sob{\til{\psi}_{\al}}{\dot{H}^\be}\ls h^{1-\be}$, for all $\be\geq0$.
	\end{enumerate}
\end{prop}

\begin{rmk}\label{rmk:sob:char}
When $\s\geq0$, let $[\s]$ denote the greatest integer such that $[\s]\leq\s$.  Then define
	\begin{align}\label{sob:hom:char}
		\Sob{\phi}{\til{\dot{H}}^\s}^2:=\sum_{0<|\mathbf{k}|\leq [\s]}\Sob{\bdy^{\mathbf{k}}\phi}{L^2}^2+\sum_{|\mathbf{k}|=[\s]}\Sob{\bdy^{\mathbf{k}}\phi}{\dot{H}^{\s-[\s]}}^2,\quad \mathbf{k}\in\N^2\cup\{\mathbf{0}\},\quad \bdy^{\mathbf{k}}=\bdy_x^{k_1}\bdy_y^{k_2}.
	\end{align}
where for $0<\be<1$, we define
	\begin{align}\label{sob:char}
		\Sob{\phi}{\til{\dot{H}}^\be}^2:=\int_{\T^2}\int_{[-{\pi},{\pi})^2}\frac{\abs{\phi(x+y)-\phi(y)}^2}{\abs{x}^{2+2\be}}\ dx\ dy.
	\end{align}
Then $\Sob{\cdotp}{\til{\dot{H}}^\s}$ is equivalent to \req{Hdper} when $\s\geq0$ (cf. \cite{adams, benyi:oh} and Proposition \ref{lem:ben:oh}).  Indeed, there exists an absolute constant, $C>0$, such that for all $\phi\in H_{per}^\s(\T^2)$ with $\s\geq0$, we have
	\begin{align}\notag
		C^{-1}\Sob{\phi}{\til{\dot{H}}^\s}^2\leq\Sob{\phi}{\dot{H}^\s}^2\leq C\Sob{\phi}{\til{\dot{H}}^\s}^2.
	\end{align}
Therefore, to see (iv), let $\til{\psi}_\al^{\mathbf{k}}:=(\bdy^{\mathbf{k}}\rho)_{h/10}*\psi_\al$, for $\mathbf{k}\in\N^2\cup\{\mathbf{0}\}$.  Observe that
	\begin{align}\label{dk:psial}
		\bdy^{\mathbf{k}}\til{\psi}_\al=(\bdy^{\mathbf{k}}\rho_{h/10})*\til{\psi}_\al=\frac{10^{|\mathbf{k}|}}{h^{|\mathbf{k}|}}\til{\psi}_\al^{\mathbf{k}},
	\end{align}
and by Young's convolution inequality we have
	\begin{align}\label{dk:convo}
		\Sob{\til{\psi}_\al^{\mathbf{k}}}{L^2}\leq C_{\mathbf{k}}h,
	\end{align}
where $C_{\mathbf{k}}$ depends on $\bdy^{\mathbf{k}}\rho$, but not $h$.  On the other hand, by \req{sob:char} one can show that
	\begin{align}\label{dk:fracsob}
		\Sob{\til{\psi}_\al^{\mathbf{k}}}{H^{\be-[\be]}}\leq Ch^{1-\be+[\be]},
	\end{align}
where $C$ depends only on $\rho, \be$, but not $h$.  Thus, from \req{sob:hom:char}, \req{dk:psial}, \req{dk:convo}, and \req{dk:fracsob} that for $\be>0$, we have
	\begin{align}
		\Sob{\til{\psi}_\al}{\dot{H}^\be}^2&\leq C\sum_{0<|\mathbf{k}|\leq[\be]}\Sob{\bdy^{\mathbf{k}}\til{\psi}_\al}{L^2}^2+C\sum_{|\mathbf{k}|=[\be]}\Sob{\bdy^{\mathbf{k}}\til{\psi}_\al}{H^{\be-[\be]}}^2\notag\\
			&\leq C\sum_{0<|\mathbf{k}|\leq[\be]}|h|^{-2|\mathbf{k}|}\Sob{\til{\psi}_\al^{\mathbf{k}}}{L^2}^2+C\sum_{|\mathbf{k}|=[\be]}|h|^{-2[\be]}\Sob{\til{\psi}_\al^{\mathbf{k}}}{H^{\be-[\be]}}^2\leq Ch^{2-2\be},
	\end{align}
as desired.

\end{rmk}



For $\phi\in L^1_{loc}(\Om)$, define
	\begin{align}\notag
		\phi_{Q}=\frac{1}{a({Q})}\int_{Q}\phi(x)\ dx\quad\text{and}\quad\til{\phi}_{Q_\al}=\frac{1}{\til{a}({{Q}_\al})}\int_{\T^2}\phi(x)\til{\psi}_\al(x)\ dx,
	\end{align}
where $a({Q})$ denotes the area of $Q$ and
	\begin{align}\label{mod:avg}
		\til{a}({{Q}_\al}):=\int_{\T^2}\til{\psi}_\al(x)\ dx.
	\end{align}
At this point, let us emphasize that $\til{\psi}_\al$ are \textit{non-negative} for each $\al\in\mathcal{J}$.  Observe that for each $\al\in\mathcal{J}$, there exists an absolute constant $c>0$, independent of $h, \al,\eps$, such that
	\begin{align}\label{equiv:cubes}
		c^{-1}\leq \frac{\til{a}{({Q_\al})}}{{a}({Q})},  \frac{a{({Q})}}{{a}({\hat{Q}_\al})},  \frac{a{({Q})}}{{a}({Q_\al(\eps)})}\leq c,\quad Q\in\{Q_\al, Q_\al(\eps), \hat{Q}_\al\}.
	\end{align}

Finally, we define the smooth volume element interpolant by
	\begin{align}\label{vol:elts}
		\mathcal{I}_h(\phi):=\sum_{\al\in\mathcal{J}}\til{\phi}_{Q_\al}\til{\psi}_\al
	\end{align}
and the ``shifted" smooth volume element interpolant by
	\begin{align}\label{vol:elts:shift}
		{I}_h(\phi):=\sum_{\al\in\mathcal{J}}\til{\phi}_{Q_\al}\bar{\psi}_\al,\quad \bar{\psi}_\al:=\til{\psi}_{\al}-\lb\til{\psi}_\al\rb.
	\end{align}

Observe that we have the following relation between $\mathcal{I}_h$ and $I_h$.

\begin{lem}\label{lem:Ih}
Let $\mathcal{I}_h, I_h$ be given by \req{vol:elts}, \req{vol:elts:shift}.  Let $\phi\in L^1_{loc}(\T^2)$.  Then
	\begin{enumerate}[(i)]
		\item $I_h\phi=\mathcal{I}_h\phi-\lb\mathcal{I}_h\phi\rb$;
		\item $\lb I_h\phi\rb=0$;
		\item $\Lam^\be I_h\phi=\Lam^\be\mathcal{I}_h\phi$, for $\be>0$.
	\end{enumerate}
\end{lem}

Now let us prove Proposition \ref{thm:Jh} when $J_h$ is given by either $\mathcal{I}_h$ or $I_h$, as defined by \req{vol:elts}, \req{vol:elts:shift}, respectively.

\begin{proof}[Proof of Proposition \ref{thm:Jh}: Part I]
We will establish  (Property 0.1)-(Property 0.3) for $J_h$ given by either \req{vol:elts} or \req{vol:elts:shift}.  It will suffice to consider $J_h=\mathcal{I}_h$ given by \req{vol:elts}.  Indeed, by Lemma \ref{lem:Ih}, the fact that
	\begin{align}\label{avg:lp}
		\Sob{\lb \phi\rb}{L^p}\leq \Sob{\phi}{L^p},\quad \phi\in L^p(\T^2),\quad 1\leq p\leq\infty,
	\end{align}
and the triangle inequality, we have that the properties (Property 0.1)-(Property 0.3) applied to $J_h=\mathcal{I}_h$ easily imply the corresponding properties for $J_h=I_h$ given by \req{vol:elts:shift}.

Suppose then that $J_h=\mathcal{I}_h$.  Observe that for each $x\in\T^2$ and $\al\in\mathcal{J}$, we have that $n_\al(x):=\#\{Q_i(\eps):\til{\psi}_\al(x)\neq0,\ \text{for some}\ x\in Q_\al(\eps)\}$ is independent of $h$.  In particular, $\sup_{\al\in\mathcal{J}}\sup_{{x\in \Om}}n_\al(x)=n_0$, for some fixed positive integer $n_0$, independent of $N, h$.  It follows that
	\begin{align}\label{orthog1}
		\left(\sum_{\al\in\mathcal{J}}\til{\phi}_{Q_\al}\til{\psi}_\al\right)^p\leq C^p\sum_{\al\in\mathcal{J}}\til{\phi}_{Q_\al}^p\til{\psi}_\al^p,
	\end{align}
for some absolute constant $C>0$ depending only on $n_0$.

We prove property (Property 0.1).  For $1\leq p\leq\infty$, it follows from \req{equiv:cubes}, \req{orthog1}, and H\"older's inequality that
	\begin{align}
		\Sob{\mathcal{I}_h\phi}{L^p}^p&\leq C^p\sum_{\al\in\mathcal{J}}|\til{\phi}_{Q_\al}|^p\Sob{\til{\psi}_\al}{L^p}^p\leq C^{p}\sum_{\al\in\mathcal{J}}\left(\frac{a(Q_\al(\eps))}{\til{a}(Q_\al)}\right)^p\Sob{\phi}{L^p(Q_\al(\eps))}^p\lesssim C^p\Sob{\phi}{L^p}^p.\notag
	\end{align}
for some absolute constant $C>0$ that depends on $n_0$.

Next, we prove (Property 0.2).  Suppose $p\geq2$.    It follows from \req{equiv:cubes}, \req{orthog1}, and the Cauchy-Schwarz inequality that
	\begin{align}
		\Sob{\mathcal{I}_h\phi}{L^p}^p&\leq C^p \sum_{\al\in\mathcal{J}}|\til{\phi}_{Q_\al}|^p\Sob{\til{\psi}_\al}{L^p}^p\notag\\
						&\leq C^p\sum_{\al\in\mathcal{J}}\frac{1}{\til{a}(Q_\al)^{p}}\Sob{\phi}{L^2(Q_\al(\eps))}^p\Sob{\til{\psi}_\al}{L^2}^p\Sob{\til{\psi}_\al}{L^p}^p\notag\\
						&\leq C^p\Sob{\phi}{L^2}^{p-2}\sum_{\al\in\mathcal{J}}\left(\frac{a(Q_\al(\eps))}{\til{a}(Q_\al)}\right)^{p/2}\left(\frac{a(Q_\al(\eps))}{\til{a}(Q_\al)^{p/2}}\right)\Sob{\phi}{L^2(Q_\al(\eps))}^{2}\notag\\
						&\ls C^ph^{2-p}\Sob{\phi}{L^2}^p,\notag
	\end{align}
for some absolute constant $C>0$, depending on $n_0$.

To prove (Property 0.3), it follows from Proposition \ref{prop:pou} $(iv)$, \req{equiv:cubes}, and the Cauchy-Schwarz inequality that
	\begin{align}\notag
		\Sob{\mathcal{I}_h\phi}{\dot{H}^\be}^2\leq \sum_{\al\in\mathcal{J}}\til{\phi}_{Q_\al}^2\Sob{\til{\psi}_{\al}}{\dot{H}^\be}^2\leq Ch^{-2\be}\sum_{\al\in\mathcal{J}}a({Q_\al})\til{\phi}_{Q_\al}^2\leq Ch^{-2\be}\Sob{\phi}{L^2}^2,
	\end{align}
as desired.  Note that the absolute constant above depends on $\til{\psi}_\al$, but is independent of $h$.  
\end{proof}

To prove Part II of Proposition \ref{thm:Jh}, we will require the following two results, the first of which is a fractional Poincar\'e-type inequality.  The second provides an alternate characterization of Sobolev norms.

\begin{lem}[\cite{hurri-syrj}]\label{lem:fp}
Let $Q\sub\R^2$ be a closed square.  Let $1\leq q\leq p<\infty$ and $\de,\rho\in(0,1)$.  Then for $\phi\in L^p(Q)$, we have
	\begin{align}\notag
		\frac{1}{{a(Q)}}\int_Q\abs{\phi(x)-\phi_Q}^q\ dx\lesssim{a(Q)}^{q(\de/2-1/p)}\left(\int_Q\int_{Q\cap B(x,\rho\abs{Q}^{1/2})}\frac{\abs{\phi(x)-\phi(y)}^p}{\abs{x-y}^{2+\de p}}\ dy\ dx\right)^{q/p},
	\end{align}
where the suppressed absolute constant is independent of $\phi$.
\end{lem}

\begin{prop}[\cite{benyi:oh}]\label{lem:ben:oh}
Let $0<\be<1$.  Then for $\phi\in\dot{H}^\be_{per}(\T^2)$, we have
	\begin{align}\notag
		\Sob{\phi}{\dot{H}^\be(\T^2)}^2\sim\int_{\T^2}\int_{[-{\pi},{\pi})^2}\frac{\abs{\phi(x+y)-\phi(y)}^2}{\abs{x}^{2+2\be}}\ dx\ dy.
	\end{align}
\end{prop}

When then have the following corollary.

\begin{coro}\label{coro:poin}
Let $0<\de<1$ and $Q\sub\R^2$ a closed square.  Then for $\phi\in{H}_{per}^\de(Q)$, we have
	\begin{align}\notag
		\int_Q\abs{\phi(x)-\phi_Q}^2\ dx\lesssim\abs{Q}^{2\de}\Sob{\phi}{\dot{H}^\de(Q)}^2.
	\end{align}
\end{coro}

\begin{proof}
Let $Q\sub\R^2$ be a closed square and $x_0$ denote its center and $Q_0$ denote the same square, but centered at the origin.  Let $\phi\in {H}^\de_{per}(Q)$ and $\rho\leq1/4$ and let $\phi_{x_0}(x)=\phi(x+x_0)$.  Then observe by translating and  rescaling, we have
	\begin{align}
		\int_Q\int_{Q\cap B(x,\rho\abs{Q}^{1/2})}\frac{\abs{\phi(x)-\phi(y)}^p}{\abs{x-y}^{2+\de p}}\ dy\ dx&\leq\int_{Q_0}\int_{B(0,\rho\abs{Q}^{1/2})}\frac{\abs{\phi_{x_0}(x)-\phi_{x_0}(x+y)}^p}{\abs{y}^{2+\de p}}\ dy\ dx\notag\\
					&=\int_{Q_0}\int_{B(0,\rho\abs{Q}^{1/2})}\frac{\abs{\phi_{x_0}(x+y)-\phi_{x_0}(y)}^p}{\abs{x}^{2+\de p}}\ dx\ dy\notag\\
					&=\frac{\abs{Q}^{1-\de p/2}}{(2\pi)^{2-\de p}}\int_{\T^2}\int_{B(0,\rho)}\frac{\abs{\til{\phi}_{x_0}(x+y)-\til{\phi}_{x_0}(y)}^p}{\abs{x}^{2+\de p}}\ dx\ dy\notag\\	
						&\leq \frac{\abs{Q}^{1-\de p/2}}{(2\pi)^{2-\de p}}\int_{\T^2}\int_{[-{\pi},{\pi})^2}\frac{\abs{\til{\phi}_{x_0}(x+y)-\til{\phi}_{x_0}(y)}^p}{\abs{x}^{2+\de p}}\ dx\ dy,\notag
	\end{align}
where $\til{\phi}(x)=\phi((\abs{Q}^{1/2}/(2\pi))x)$.  Thus, by setting $p=q=2$, then applying Lemma \ref{lem:fp} and Proposition \ref{lem:ben:oh}, we obtain
	\begin{align}\notag
		\int_Q\abs{\phi(x)-\phi_Q}^2\ dx\lesssim(2\pi)^{2\de-2}\Sob{\til{\phi}_{x_0}}{\dot{H}^\de(\T^2)}^2\sim\abs{Q}^{2\de}\Sob{\phi}{\dot{H}^\de(Q)}^2,
	\end{align}
as desired.
\end{proof}

The next result adapts Corollary \ref{coro:poin} to modified local spatial averages.   In particular, given a square $Q\sub\R^2$ and $\eps>0$, define $Q(\eps)=Q+B(0,\eps)$.  Suppose that $\til{\psi}\in C^\infty(\R^2)$ is an arbitrary non-negative function satisfying $0\leq\til{\psi}\leq1$,  $\spt\psi\sub Q(\eps)$, and $\til{\psi}|_Q>0$.  Then, given $\phi\in L^1_{loc}(\R^2)$, we define
	\begin{align}\label{mod:avg:2}
		\til{a}(Q):=\int\til{\psi}(x)\ dx\quad\text{and}\quad \til{\phi}_Q:=\frac{1}{\til{a}(Q)}\int\phi(x)\til{\psi}(x)\ dx.
	\end{align}

\begin{coro}\label{coro:poin2}
Let $q\geq1$, $0<\de<1$, and $Q\sub\R^2$ a closed square.  Then for $\phi\in{H}^\de_{per}(Q)$, we have
	\begin{align}\label{coro:poin2:ineq}
		\int_Q\abs{\phi(x)-\til{\phi}_{Q_\al}}^q\ dx\lesssim(a({Q_\al})+\eps^2)^{2\de}\Sob{\phi}{H^\de(Q)}^2.
	\end{align}
\end{coro}

\begin{proof}
First observe that
	\begin{align}\notag
		\phi_Q-\til{\phi}_Q=\frac{1}{\til{a}(Q)}\int(\phi_Q-\phi(x))\til{\psi}(x)\ dx.
	\end{align}
Then by H\"older's inequality, we have
	\begin{align}\notag
		\abs{\phi_Q-\til{\phi}_Q}^q\leq\frac{1}{\til{a}(Q)}\int\abs{\phi(x)-\phi_Q}^q\psi(x)\ dx\leq\frac{1}{\til{a}(Q)}\int_{Q(\eps)}\abs{\phi(x)-\phi_Q}^q\ dx.
	\end{align}
It follows then from Minkowski's inequality and convexity that
	\begin{align}\notag
		\int_Q|\phi(x)-\til{\phi}_Q|^q\leq C^q\left(\int_Q|\phi(x)-\phi_Q|^q\ dx+\frac{a(Q)}{\til{a}(Q)}\int_{Q(\eps)}|\phi(x)-\til{\phi}_Q|^q\ dx\right).
	\end{align}
Therefore, by \req{equiv:cubes} and Corollary \ref{coro:poin}, we obtain \req{coro:poin2:ineq}.
\end{proof}

Finally, we are ready to complete the proof of Proposition \ref{thm:Jh} when $J_h$ is given by \req{vol:elts} or \req{vol:elts:shift}.

\begin{proof}[Proof of Proposition \ref{thm:Jh}: Part II]
First suppose that $J_h=\mathcal{I}_h$ is given by \req{vol:elts}.  The case $\be=1$ follows from the classical Poincar\`e inequality, so let $\be\in(0,1)$ and $\phi\in\dot{H}_{per}^\be(\T^2)$. 
Thus, by Proposition \ref{prop:pou}, \req{orthog1}, and Corollary \ref{coro:poin2}, it follows that
	\begin{align}\label{poin:pos}
		\Sob{\phi-\mathcal{I}_h\phi}{L^2}^2\lesssim\sum_{\al\in\mathcal{J}}\Sob{\phi-\til{\phi}_{Q_\al}}{L^2(Q_{\al}(\eps))}^2\lesssim h^{2\be}\sum_{\al\in\mathcal{J}}\Sob{\phi}{\dot{H}^\be(Q_\al(\eps))}^2\sim h^{2\be}\Sob{\phi}{\dot{H}^\be}^2,
	\end{align}
which proves (Property 1.1).  To establish (Property 1.2), first observe that for $g,h\in L^1_{loc}(\T^2)$ we have
	\begin{align}\notag
		\lb(h-\til{h}_{Q_\al})\til{\psi}_\al,\til{g}_{Q_\al}\rb=\til{g}_{Q_\al}\int h(x)\psi_\al(x)\ dx-\til{h}_{Q_\al}\til{a}(Q_{\al})\til{g}_{Q_{\al}}=0,
	\end{align}
and by symmetry
	\begin{align}\notag
	\lb\til{h}_{Q_{\al}},(g-\til{g}_{Q_{\al}})\til{\psi}_\al\rb=0.
	\end{align}
It  then follows  from this and Proposition \ref{prop:pou} (ii) that for $g\in L^2_{loc}(\T^2)$, we have
	\begin{align}\label{calIh:dual}
		\lb\phi-\mathcal{I}_h\phi, g\rb&=\sum_{\al\in\mathcal{J}}\lb \phi-\til{\phi}_{Q_{\al}}, (g-\til{g}_{Q_{\al}})\til{\psi}_\al\rb=\sum_{\al\in\mathcal{J}}\lb \phi, (g-\til{g}_{Q_{\al}})\til{\psi}_\al\rb=\lb\phi, g-\mathcal{I}_hg\rb.
	\end{align}
Thus, given $g\in \dot{H}^{\be}_{per}(\T^2)$, it follows from \req{calIh:dual}, Proposition \ref{prop:pou}, and the Cauchy-Schwarz inequality that 
	\begin{align}\notag
		\abs{\lb\phi-\mathcal{I}_h\phi, g\rb}=\sum_{\al\in\mathcal{J}}\Sob{\phi}{L^2(Q_\al(\eps))}\Sob{(g-\til{g}_{Q_\al})\til{\psi}_\al}{L^2(Q_\al(\eps))}\ls\Sob{\phi}{L^2}\left(\sum_{\al\in\mathcal{J}}\Sob{(g-\til{g}_{Q_\al})\til{\psi}_\al}{L^2(Q_\al(\eps))}^2\right)^{1/2}.
	\end{align}
Then Corollary \ref{coro:poin2} implies
	\begin{align}\notag
		\Sob{g-\til{g}_{Q_\al}}{L^2(Q_\al)}^2\ls h^{2\be}\Sob{g}{H^\be(Q_\al)}^2.
	\end{align}
Therefore, by duality we have
	\begin{align}\label{calIh:result}
		\Sob{\phi-\mathcal{I}_h\phi, g}{\dot{H}^{-\be}}=\sup_{\Sob{g}{\dot{H}^\be}=1}\abs{\lb\phi-\mathcal{I}_h\phi, g\rb}\ls h^\be\Sob{\phi}{L^2},
	\end{align}
which is precisely (Property 1.2).

Now let $J_h=I_h$ be given by \req{vol:elts:shift}.  To show that (Proprety 1.1) holds, first observe that $\mathcal{I}_h1=1$ and $I_h1=0$.  Given $\phi\in L^1_{loc}(\T^2)$ such that $\lb\phi\rb=0$, it follows from the fact that $\mathcal{I}_h, I_h$ are linear and Lemma \ref{lem:Ih} (i) that
	\begin{align}\label{Ih:Ih}
		\phi-I_h\phi&=\phi-\lb\phi\rb-I_h(\phi-\lb\phi\rb)\notag\\
				&=(\phi-\mathcal{I}_h\phi)+\mathcal{I}_h\lb\phi\rb-\lb\mathcal{I}_h\lb\phi\rb\rb-\lb\phi-\mathcal{I}_h\phi\rb\notag\\
				&=(\phi-\mathcal{I}_h\phi)-\lb\phi-\mathcal{I}_h\phi\rb.
	\end{align}
Thus, (Property 1.1) for $J_h=I_h$ and $\phi\in V_\be$ follows from Minkowski's inequality, \req{avg:lp}, and \req{poin:pos}.  To see (Property 1.2) for $J_h=I_h$, simply observe that
if $\lb g\rb=0$, then \req{Ih:Ih}  implies that
	\begin{align}\label{Ih:calIh:dual}
		\lb \phi-I_h\phi, g\rb=\lb \phi-\mathcal{I}_h\phi,g\rb-\lb \lb\phi-\mathcal{I}_h\phi\rb, g\rb=\lb \phi-\mathcal{I}_h\phi, g\rb.
	\end{align}
Recall that Lemma \ref{lem:Ih} (ii) shows that $I_h\phi\in\mathcal{Z}$ for any $\phi\in L^1_{loc}(\T^2)$.  In particular, $\phi-I_h\phi\in\mathcal{Z}$.  Thus, given $\phi\in {L}^2_{per}(\T^2)\cap\mathcal{Z}$, it follows from duality and \req{Ih:calIh:dual} that
	\begin{align}\notag
		\Sob{\phi-{I}_h\phi}{\dot{H}^{-\be}}=\sup_{\substack{g\in V_\be\\ \Sob{g}{\dot{H}^\be}=1}}\abs{\lb\phi-{I}_h\phi, g\rb}=\sup_{\substack{g\in V_\be\\ \Sob{g}{\dot{H}^\be}=1}}\abs{\lb\phi-\mathcal{I}_h\phi, g\rb}.
	\end{align}
Arguing as we did for \req{calIh:result}, we have that $J_h=I_h$ satisfies (Property 1.2), as desired.
\end{proof}

\subsection{Modal projection (Type II)}\label{sect:LWP}
Here we let $J_h$ be given by projection onto Fourier modes.  The projection can be given by either rough or smooth cut-offs in the frequency side.  The ``rough projection" will be given by convolution with the square Dirichlet kernel, while the ``smooth projection" will be given by Littlewood-Paley projection.  As in the previous section, we work with rescaled variables, so that the $2\pi$-periodic box is given by $\Om=\T^2=[-\pi,\pi]^2$.

\subsubsection*{Rough modal projection}

Let $N>0$.  For $k\in\Z^2$, $k=(k_1,k_2)$,  denote by $\hat{\phi}(k)$ the $k$-th Fourier wave-number and define the ``rough modal projection" by $P_N$ by
	\begin{align}\label{modes}
		(P_N\phi)(x_1,x_2):=(D_{{N}}*\phi)(x_1,x_2),
	\end{align}
where
	\begin{align}\label{dirichlet}
		D_{{N}}(x_1,x_2):=\sum_{\abs{k_1}\leq {N}}\sum_{\abs{k_2}\leq{N}}e^{ik\cdotp x},
	\end{align}
is the two-dimensional Dirichlet kernel.  In particular, we have
	\begin{align}
		D_N(x_1,x_2)=D_N(x_1)D_N(x_2),\quad D_N(x)=\begin{cases}\frac{\sin((N+1/2)x)}{\sin(x/2)}, & x\in\T^2\smod\{\mathbf{0}\},\\
			2N+1,& x=\mathbf{0}.
		\end{cases}
	\end{align}
	
Let us now prove Proposition \ref{thm:Jh} with
	\begin{align}\label{rough:proj}
		J_h:=P_N,\quad h=(2\pi)/{N},\quad N\geq16.
	\end{align}

\begin{proof}[Proof of Proposition \ref{thm:Jh}]
It is classical that $J_h$ defined by \req{rough:proj} this way satisfies (Property 0.1) with constant independent of $h$ (cf. \cite{graf}).  One also has the following estimate on the Dirichlet kernel for $q\in(1,\infty)$ (cf. \cite{graf}):
	\begin{align}\label{dirichlet:est}
		\Sob{D_N}{L^q(\T^2)}\sim (2{N}+1)^{2/q'},
	\end{align}
where $q,q'$ are H\"older conjugates and the suppressed constant depends on $q$.

To show that $P_N$ satisfies (Property 0.2), we apply Young's convolution inequality and \req{dirichlet:est} to obtain
	\begin{align}
		\Sob{P_N\phi}{L^p}^p&\leq \Sob{D_N}{L^{2p/(p+2)}}^p\Sob{\phi}{L^2}^p\lesssim(2{N}+1)^{p-2}\Sob{\phi}{L^2}^p\lesssim h^{2-p}\Sob{\phi}{L^2}^p.\notag
	\end{align}

To prove property (Property 0.3), we apply Plancherel and estimate as follows
	\begin{align}\notag
		\Sob{P_N\phi}{\dot{H}^\al}^2\leq CN^{2\al}\sum_{\abs{k_j}\leq {N}}\abs{\hat{\phi}(k)}^2=C N^{2\al}\Sob{\phi}{L^2}\leq Ch^{-2\al}\Sob{\phi}{L^2}.
	\end{align}

Clearly, for $\be\geq0$, $\Lam^\be I_h=I_h\Lam^\be$ by the Plancherel theorem, which proves property (Property 2.2).

To prove property (Property 2.1), let $\be>\al$.  Then from $(v)$, the Cauchy-Schwarz inequality, and the Plancherel theorem, it follows that
	\begin{align}\notag
		\Sob{\phi-I_h\phi}{\dot{H}^\al}^2\lesssim\sum_{\abs{k}>N}\abs{k}^{2\al}\abs{\hat{\phi}(k)}^2\lesssim\sum_{\abs{k}>{N}}\abs{k}^{2(\al-\be)}\abs{k}^{2\be}\abs{\hat{\phi}(k)}^2\lesssim N^{\al-\be}\Sob{\Lam^\be\phi}{L^2}^2\lesssim h^{\be-\al}\Sob{\phi}{\dot{H}^\be}^2.
	\end{align}
\end{proof}

\subsubsection*{Smooth modal projection}\label{sect:LWP}
We define this projection  by the Littlewood-Paley decomposition.  We presently give a brief review of this decomposition.  More thorough treatments can be found in \cite{bcd, danchin:notes, runst:sick, workman}.  We state the decomposition for ${\R^2}$ for convenience, but point out that it is also valid in the case $\T^2$ for periodic distributions.  In particular, the Bernstein inequality (Proposition \ref{bern}) stated below also hold in $\T^2$ provided that that one redefine the Littlewood-Paley blocks, $\lpk$, in a suitable manner (cf. \cite{danchin:notes}).

Let $\psi_0$ be a smooth, radial bump function such that $\psi_0(\xi)=1$ when $[\abs{\xi}\leq 1/4]\sub{\R^2}$, and
	\begin{align}\label{psi:0}
		0\leq\psi_0\leq1\ \text{and}\ \spt\psi_0=[\abs{\xi}\leq 1/2].
	\end{align}
Define ${\phi}_0(\xi):=\psi_0(\xi/2)-\psi_0(\xi)$.  Observe that
	\begin{align}\notag
		0\leq{\phi}_0\leq 1\ \text{and}\ \spt{\phi}_0=[1/4\leq\abs{\xi}\leq1]
	\end{align}
Now for each integer $j\geq0$, define
	\begin{align}\label{psi:phi}
		{\phi}_j(\xi):={\phi}_0(\xi2^{-j}).
	\end{align}
Then, in view of the above definitions, we clearly have
	\begin{align}\label{lp:loc}
		\spt{\phi}_j=[2^{j-2}\leq\abs{\xi}\leq2^{j}].
	\end{align}
If we let ${\phi}_{-1}:=\psi_0$ and ${\phi}_{j}\equiv0$ for $j<-1$, observe that
	\begin{align}\label{pou}
		\sum_{j\in\Z}{\phi}_j(\xi)=1,\ \text{for }\ \xi\in{\R^2}.
	\end{align}
One can then define
	\begin{align}\label{lp:block:def}
		\lpk g:={\phi}_k*g, \quad \til{\lp}_kg:=\sum_{|k-\ell|\leq 2}\lpl g,\quad S_kg:=\sum_{\ell\leq k}\lpl g,\quad T_k:=I-S_k,
	\end{align}
where ${\phi}_k:=\check{\phi}_k$ is the inverse Fourier transform of ${\phi}_k$.  We call the operators, $\lpk$, Littlewood-Paley projections.

One can show that \req{pou} implies that
	\begin{align}\label{lp:dc}
		g=\sum_{j\geq-1}\lpj g\quad\text{for all}\ g\in \Sch'({\R^2}),
	\end{align}
where $\Sch'({\R^2})$ is the space of tempered distributions over ${\R^2}$.

For $N>0$, define $J_h$ by
	\begin{align}\label{smooth:proj}
		J_h:=S_N,\quad h=2^{-N}.
	\end{align}
 That (Property 0.1)-(Property 0.3) and (Property 2.1) and (Property 2.2) are satisfied by $I_h$ defined in this way follows from the Bernstein inequalities (cf. \cite{bcd}).

\begin{prop}[Bernstein inequalities]\label{bern}
Let $1\leq p\leq q\leq\infty$ and $g\in\Sch'({\R^2})$.  Then for $\be\in\R$ and $j\geq-1$ we have
	\begin{align}\notag
		\begin{split}
		&\Sob{\Lam^\be \lpj g}{L^q}\sim2^{j\be}\Sob{\lpj g}{L^p},\quad \Sob{\lpj g}{L^q}\lesssim 2^{2j(1/p-1/q)}\Sob{\lpj g}{L^p}.
		\end{split}
	\end{align}
For $\s\geq0$ and $j\geq-1$, we have
	\begin{align}\notag
		 \Sob{\Lam^\be S_jg}{L^q}\lesssim 2^{\be j+2j(1/p-1/q)}\Sob{S_j g}{L^p},
	\end{align}
where the suppressed absolute constants depend only on $\be, \check{\phi}_0, \check{\psi}_0$.
\end{prop}

Indeed, let us prove Proposition \ref{thm:Jh} for $J_h=S_N$ given by \req{smooth:proj}.

\begin{proof}[Proof of Proposition \ref{thm:Jh}]

(Property 0.1) and (Property 0.2) follow immediately from the Bernstein inequalities.  For property (Property 0.3), simply observe that for $\be\geq0$, we have
	\begin{align}\notag
		\Sob{S_N\phi}{H^\be}^2\lesssim2^{2\be N}\Sob{\phi}{L^2}^2.
	\end{align}
To prove property (Property 2.1), observe that for $\al,\be\in\R$, with $\be\geq\al$, we have
	\begin{align}
		\Sob{\phi-S_N\phi}{H^\al}^2\sim\sum_{j\geq N+1}2^{-2j(\be-\al)}2^{2\be j}\Sob{\lpj \phi}{L^2}^2\lesssim 2^{-2N(\be-\al)}\Sob{\phi}{\dot{H}^\be}^2.\notag
	\end{align}
Also, (Property 2.2) holds simply by applying the Fourier convolution theorem.  Therefore, $J_h$ given by \req{smooth:proj} is of Type II.
\end{proof}

\subsection*{\bf Acknowledgments}  The authors would like to thank the Institute of Pure and Applied Mathematics at UCLA, where part of this work was performed.  The authors would also like to thank Cecilia Mondaini for insightful discussions in the course of this work.  M.S Jolly was supported by NSF grant DMS-1418911 and the Leverhulme Trust grant VP1-2015-036.  The work of E.S.T. was supported in part by the ONR grant N00014-15-1-2333 and the NSF grants DMS-1109640 and DMS-1109645.

\bibliographystyle{plain}


\end{document}